\documentclass[12pt]{amsart}

\usepackage{amscd,amsfonts,amsmath,amsthm,amssymb,verbatim,mathrsfs}
\usepackage{latexsym}
\usepackage{latexsym,amssymb,amsmath,amsfonts,epsfig,textcomp}
\usepackage{color,setspace}
\usepackage[active]{srcltx}

\numberwithin{equation}{section}     

\def\sqr#1#2{{\vcenter{\hrule height.#2pt
        \hbox{\vrule width.#2pt height#1pt \kern#1pt
                \vrule width.#2pt}
        \hrule height.#2pt}}}

\def\split{\genfrac{}{}{0pt}1}

\newtheorem{Theorem}{\bf Theorem}[section]
\newtheorem{Lemma}[Theorem]{\bf Lemma}
\newtheorem{Corollary}[Theorem]{\bf Corollary}
\newtheorem{Proposition}[Theorem]{\bf Proposition}

\newtheorem{Remark/Definition}[Theorem]{\bf Remark/Definition}

\theoremstyle{definition}
\newtheorem{Remark}[Theorem]{Remark}
\newtheorem{Example}[Theorem]{Example}

\definecolor{dgreen}{rgb}{0.24, 0.17, 0.12}

\newcommand{\s}{\; | \;}

\newcommand{\ul} \underline

\newcommand{\mtx}{\left[ \begin{matrix}}
\newcommand{\mtxend}{\end{matrix}\right]}
\DeclareMathOperator{\reg}{{reg}}
\newcommand{\ffi}{\varphi}

\def\ZZ{\mathbb Z}
\def\PP{\mathbb P}

\newcommand{\cF}{\mathcal F}
\newcommand{\T}{{\bf T}}
\newcommand{\<}{\prec}
\DeclareMathOperator{\ini}{in}
\DeclareMathOperator{\hht}{ht}
\DeclareMathOperator{\lk}{lk}

\newcommand{\la}{\lambda}

\newcommand{\bT}{\mathbf T}
\newcommand{\bS}{{\mathbf S}}
\newcommand{\bTlm}{{\mathbf T}_{\la - \mu}}
\newcommand{\bSlm}{{\mathbf S}_{\la - \mu}}
\newcommand{\bTwlm}{{\mathbf T}_{\widetilde{\la} - \mu}}
\newcommand{\bSwlm}{{\mathbf S}_{\widetilde{\la} - \mu}}

\newcommand{\fm}{\mathfrak m}
\newcommand{\fa}{\mathfrak a}
\newcommand{\fb}{\mathfrak b}

\begin{document}

\title[Blow-up algebras, determinantal ideals, and Dedekind-Mertens-like formulas]{Blow-up algebras, determinantal ideals, and Dedekind-Mertens-like formulas}

\author[A.\ Corso]{Alberto Corso}
\address{Department of Mathematics, University of Kentucky, Lexington, 715 Patterson Office Tower,  KY 40506,
USA}
\email{alberto.corso@uky.edu}

\author[U.\ Nagel]{Uwe Nagel}
\address{Department of Mathematics, University of Kentucky, 715 Patterson Office Tower, Lexington,  KY 40506,
USA}
\email{uwe.nagel@uky.edu}

\author[S.\ Petrovi\'c]{Sonja Petrovi\'c}
\address{Department of  Applied Mathematics, Illinois Institute of Technology, Chicago, IL 60616, USA}
\email{sonja.petrovic@iit.edu}

\author[C.\ Yuen]{Cornelia Yuen}
\address{Department of Mathematics, The State University of New York at Potsdam, Potsdam, NY 13676, USA}
\email{yuenco@potsdam.edu}


\thanks{The work for this paper was done while the second author was 
sponsored by the National Security Agency under Grant
Numbers H98230-09-1-0032 and H98230-12-1-0247, and by the Simons Foundation under grants \#208869 and \#317096. 
The third author gratefully acknowledges  partial support by  AFOSR/DARPA grant \#FA9550-14-1-0141 during the final phase of this project and thanks the University of Kentucky Math Department for its hospitality. \\
{The authors are also grateful to the referee for helpful suggestions that improved the exposition.} } 
\keywords{Blow-up algebras, determinantal ideal, Gr\"obner basis, quadratic and Koszul algebra,
Cohen-Macaulay algebra, liaison, vertex-decomposability,  Hilbert function, Castelnuovo-Mumford regularity,
Ferrers and threshold graphs, skew shapes,  reductions}

\subjclass[2010]{Primary 13C40, 13F55, 13P10, 14M25, 16S37; Secondary 14M05, 14M12, 14N10, 05E40, 05E45}

\begin{abstract} 
We investigate Rees algebras and special fiber rings obtained by blowing up specialized Ferrers ideals. This class of monomial ideals includes strongly stable monomial ideals generated in degree two and edge ideals of prominent classes of graphs. We identify the equations of these blow-up algebras. They generate determinantal ideals associated to subregions of a generic symmetric matrix, which may have holes.  Exhibiting Gr\"obner bases for these ideals and using methods from Gorenstein liaison theory, we show that these determinantal rings are normal Cohen-Macaulay domains that are Koszul, that the initial ideals correspond to vertex decomposable simplicial complexes, and we determine their Hilbert functions and Castelnuovo-Mumford regularities. As a consequence, we find explicit minimal reductions for all Ferrers and many specialized Ferrers ideals, as well as their reduction numbers. These results can be viewed as extensions of the classical Dedekind-Mertens formula for the content of the product of two polynomials.
\end{abstract}

\maketitle



\section{Introduction}

Determinantal ideals have been a classic object of investigation in algebraic geometry and commutative algebra (see, e.g., \cite{Abhy, BV, MS, Conca, EHU, HT, KMi}). In this paper we introduce a new class of determinantal ideals. They are associated to certain subregions of a generic symmetric matrix. The novelty is that the region is allowed to have holes. We show that the minors generating the ideal form a Gr\"obner basis (with respect to a suitable term order) and deduce that their quotient rings are normal Cohen-Macaulay domains that are Koszul. Using methods from liaison theory, we establish that their initial ideals are squarefree and the Stanley-Reisner ideals of vertex decomposable simplicial complexes. We also use this approach to determine the Hilbert function and Castelnuovo-Mumford regularity of the quotient rings. 

The class of determinantal ideals introduced here arises naturally in the investigation of certain blow-up algebras. In fact, these ideals describe the  equations of special fiber rings and Rees algebras when one blows up certain monomial ideals, called specialized Ferrers ideals (see \cite{CorsoNagel2}). These monomial ideals are generated by quadrics and include all strongly stable monomial ideals that are generated in degree two as well as the edge ideals of threshold and Ferrers graphs - two ubiquitous classes of graphs. Using Gr\"obner bases, we also produce explicit minimal reductions of (many specialized) Ferrers ideals.  We then show how our knowledge of the Castelnuovo-Mumford regularity allows us to determine their reduction numbers. These results can be viewed as a generalization of the classical Dedekind-Mertens content formula. Finding distinguished classes of reductions is potentially of interest in  areas as diverse as birational geometry (see, e.g., \cite{PUV, Smith}) and algebraic statistics (see below). 

\medskip

The origins of some of our results can be traced back to the Dedekind-Mertens formula. The \emph{content} $c(f)$ of a polynomial  $f = a_1 + a_2t + \cdots + a_nt^{n-1}\in R[t]$ over a commutative ring $R$ is the $R$-ideal $(a_1, a_2, \ldots, a_n)$. Generalizing  Gauss's Lemma for a PID,  Dedekind and Mertens \cite{Northcott} gave the  general content formula for the product of two polynomials $f, g \in R[t]$: 
\begin{equation}\label{contentformula}
c(f g) \cdot c(g)^{n-1} = c(f) \cdot  c(g)^{n}.
\end{equation}
In \cite{CVV}, this equation is explained  in terms of the theory of Cohen-Macaulay rings for generic polynomials $f=x_1 + \cdots + x_nt^{n-1}$ and $g=y_m +
\cdots + y_1 t^{m-1}$. By multiplying both sides of (\ref{contentformula}) by $c(f)^{n-1}$, one obtains the  `decayed' content equation
\begin{equation}\label{contentformula2}
c(f g)  \cdot [c(f) \cdot  c(g)]^{n-1} = [c(f)  \cdot c(g)]^n.
\end{equation}
By \cite{CVV}, if $n \le m$,  the exponent $n-1=\deg f$ in (\ref{contentformula2}) is the least possible. That is, $c (fg)$ is a minimal reduction of $c(f) \cdot c(g)$ with  reduction number  $\min\{ n, m \}- 1$ (see Figure~\ref{fig:fullContentFormula}). Subsequently, a combinatorial proof of the Dedekind-Mertens formula was given by Bruns and Guerrieri \cite{BG} via a study of the  Gr\"{o}bner basis of the ideal $c(fg)$. 
\begin{figure}[h!]
\includegraphics{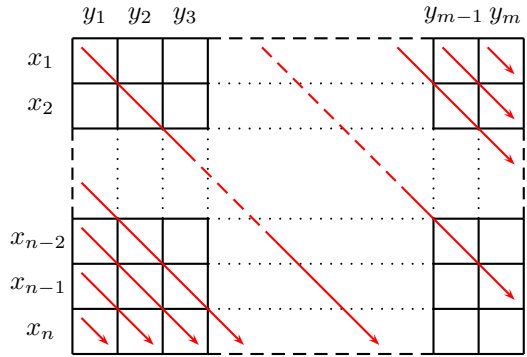}
\caption{The ideal $c(fg)$ in relation to the ideal $c(f)c(g)$.}
\label{fig:fullContentFormula}
\end{figure}
The boxes in Figure~\ref{fig:fullContentFormula} are naturally associated to the edges of a complete bipartite graph 
(with vertices $x_1,\dots,x_n$ and $y_1,\dots,y_m$). Its diagonals correspond to the generators of $c(fg)$. 
As a consequence of our results on blow-up rings,   we generalize the classical content reduction formula for a full rectangular tableau to Dedekind-Mertens-like formulas for Ferrers tableaux and skew shapes. We proceed in two steps.

\medskip 

In the first step, instead of a rectangle we consider more generally a Ferrers tableau and its corresponding Ferrers ideal. 
Any partition $\lambda=(\lambda_1, \ldots, \lambda_n)$ corresponds to a {\it Ferrers tableau} ${\mathbf T}_{\lambda}$, 
which is an array of $n$ rows of cells with $\lambda_i$  cells  in row $i$, left justified. The corresponding {\it Ferrers ideal} 
has a monomial generator corresponding to each cell in ${\mathbf T}_{\lambda}$, that is, 
\[
I_\lambda = (x_i y_j \; | \;  1 \le j \le \la_i, \ 1 \le i \le n) \subset K[x_1, \ldots, x_n, y_1, \ldots, y_m], 
\]
where $m = \la_1$. It is the edge ideal of a Ferrers graph (see Figure~\ref{fig:Ferrers tablau}). 


Ferrers graphs/tableaux have a prominent place in the literature as
they have been studied in relation to chromatic polynomials
\cite{Brenti,EvW}, Schubert varieties \cite{Ding,Develin},
hypergeometric series \cite{Haglund}, permutation statistics
\cite{Butler,EvW}, quantum mechanical operators \cite{Varvak}, and
inverse rook problems \cite{GJW,Ding,Develin}. More
generally, algebraic and combinatorial aspects of bipartite graphs
have been studied in depth (see, e.g., \cite{SVV, HH, Froberg, CorsoNagel, CorsoNagel2, NRei, DE} and the
comprehensive monographs  \cite{HH-book, vila}). In this paper we complete a study initiated in \cite{CorsoNagel} by exhibiting,  in particular, explicit minimal reductions of Ferrers ideals. 
More precisely, we show that the diagonals in any Ferrers tableau ${\mathbf T}_{\lambda}$ correspond to the generators 
of a minimal reduction $J_{\la}$ of the Ferrers ideal $I_{\la}$ (see Theorem \ref{thm:min-red-Ferrers} and Figure \ref{fig:Ferrers tablau}) 
and that $I_{\la}$  has reduction number (see Theorem \ref{thm:red number Ferrers})
\[
r_{J_{\lambda}} (I _{\lambda}) = \min \{n-1, \la_i + i - 3 \s 2 \le i \le n\}.
\]   

\begin{figure}[h!]
\includegraphics{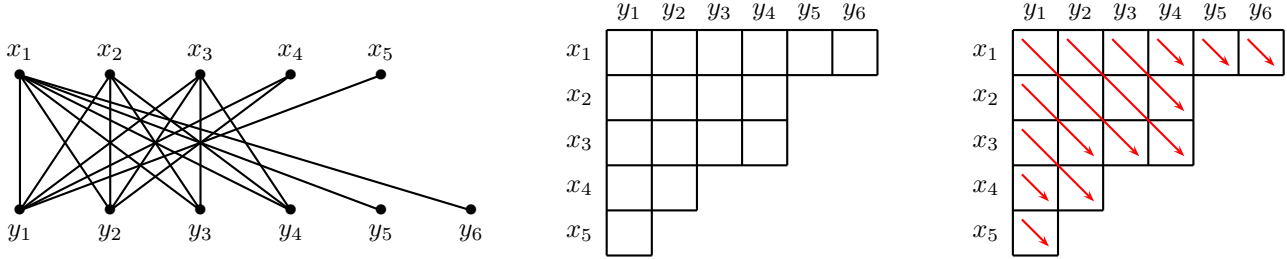}
\[
I_{\la}= (x_1y_1, x_1y_2, x_1y_3, x_1y_4, x_1y_5, x_1y_6, x_2y_1, x_2y_2,
x_2y_3, x_2y_4, x_3y_1, x_3y_2, x_3y_3, x_3y_4, x_4y_1, x_4y_2,
x_5y_1)
\]
\caption{Ferrers graph, tableau, reduction and ideal for $\lambda=(6,4,4,2,1)$.}
\label{fig:Ferrers tablau}

\end{figure}
\smallskip

In the second step, we investigate the ideals that one obtains from Ferrers ideals by \emph{specialization}, that is, by substituting $y_j \mapsto x_j$. In order to infer properties of the resulting ideals, one wants to preserve the number of generators in this process. This forces us to adjust the traditional notation.  Given a partition $\lambda = (\lambda_1,\ldots,\lambda_n)$, let $\mu = (\mu_1,\ldots,\mu_n) \in {\mathbb Z}^n$ be
a vector such that $0 \leq \mu_1 \leq \cdots \leq \mu_n <
\lambda_n$. Form a diagram ${\mathbf T}_{\lambda - \mu}$, obtained from  ${\mathbf T}_{\lambda}$ by removing the leftmost  $\mu_i$ boxes in row $i$ (see Figure \ref{fig:specializedHoles}). The ideal whose generators correspond to the cells of ${\mathbf T}_{\lambda - \mu}$ was called a  \emph{generalized Ferrers ideal} $I_{\lambda-\mu}$ in \cite{CorsoNagel2}. Thus,  
\[
I_{\lambda - \mu} := (x_i y_j \s 1 \leq i \leq n, \mu_i < j \leq
\lambda_i) \subset K[x_1,\ldots,x_n, y_1,\ldots,y_m].
\]
It is isomorphic to a Ferrers ideal. Substituting $y_j \mapsto x_j$ gives the \emph{specialized Ferrers ideal} 
\[
\overline I_{\lambda - \mu} := (x_i x_j \s 1 \leq i \leq n, \mu_i < j \leq
\lambda_i)  \subset K[x_1,\ldots,x_{\max \{n, m\}}].
\]
In order to guarantee that $I_{\lambda - \mu}$ and  $\overline I_{\lambda - \mu}$ have the same number of minimal generators we assume throughout $\mu_i \ge i-1$ for $i = 1,\ldots,n$. Thus, ${\mathbf T}_{\lambda - \mu}$ is a skew shape.  Notice that specialized Ferrers ideals are a proper generalization of Ferrers ideals, which one obtains if $\mu_1 = \cdots = \mu_n \ge n$  (see Figure \ref{fig:no new minors by symm}). 


\begin{figure}[h!]
\includegraphics{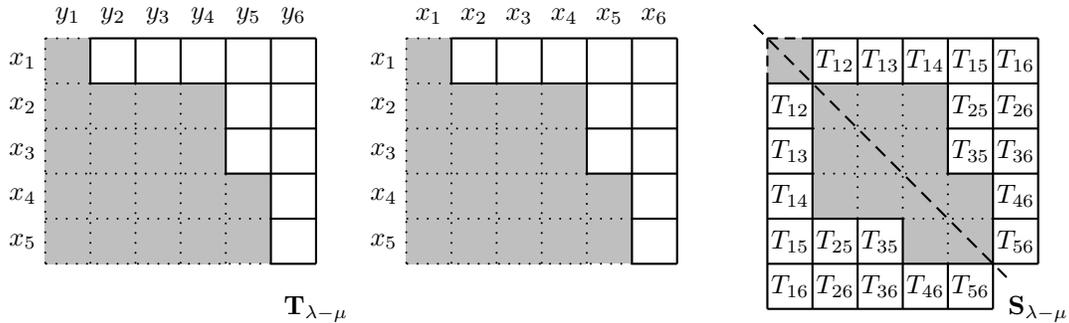}
\caption{A skew shape and its symmetrization for $\lambda =(6,6,6,6,6)$ and  $\mu =(1,4,4,5,5)$.}
\label{fig:specializedHoles} 
\end{figure}

By \cite{CorsoNagel2}, the ideal $I_{\lambda - \mu}$ and its specialization $\overline I_{\lambda - \mu}$ have closely related minimal free resolutions. Both are supported on a polyhedral cell complex whose faces can be read off from ${\mathbf T}_{\lambda - \mu}$. Thus, one wonders if  also their reductions are similarly related. Surprisingly, this is not the case.  Properties of reductions are governed by blow-up rings. In Theorem~\ref{thm:special-fiber-ring-is-ladder-determinantal}, we determine the equations of the special fiber ring of $\overline I_{\lambda - \mu}$. More precisely, these equations can be taken as 2-minors of a subregion ${\mathbf S}_{\lambda - \mu}$ of a generic matrix, where ${\mathbf S}_{\lambda - \mu}$ is obtained from ${\mathbf T}_{\lambda - \mu}$ by reflecting about the main diagonal (see Figure 
\ref{fig:specializedHoles}). Notice that, depending on $\mu$, the symmetrized tableau  ${\mathbf S}_{\lambda - \mu}$ may have holes in the middle!  A modification of this construction also allows us to identify the Rees algebra of ${\overline I}_{\lambda - \mu}$ as determinantal (see Corollary~\ref{cor:Rees-algebra}). In order to establish these results we first show that the 2-minors in the symmetrized region ${\mathbf S}_{\lambda - \mu}$ form a Gr\"obner basis of the ideal they generate (see Theorem~\ref{thm:Minors-Grob}). We then apply liaison-theoretic methods in order to analyze the corresponding initial ideals. In particular, we show that they correspond to vertex decomposable simplicial complexes  and thus are Cohen-Macaulay. To conclude, we also use  a localization argument to prove that the determinantal ideals are prime (see Proposition~\ref{prop:primality}) and determine the dimension of the special fibers ring (see Proposition~\ref{prop:dim-fibers-ring}).  


Notice that Theorem~\ref{thm:special-fiber-ring-is-ladder-determinantal} generalizes the identification of the special fiber ring of a Ferrers ideal in \cite[Proposition 5.1]{CorsoNagel}. We apply  Theorem~\ref{thm:special-fiber-ring-is-ladder-determinantal} to determine explicit minimal reductions of arbitrary Ferrers ideals (see Theorem~\ref{thm:min-red-Ferrers}) and of strongly stable specialized Ferrers ideals (see Theorem~\ref{thm:min-red-specialized-Ferr}). 
Their reduction numbers are found in Theorems \ref{thm:red number Ferrers} and \ref{thm:red number spec Ferrers}.  The latter results are based on formulas for the Hilbert functions of the special fiber rings to generalized and specialized Ferrers ideals in  Section \ref{sec:numbers}. There we also 
 establish a result that relates the reduction number to the Castelnuovo-Mumford regularity of a special fiber ring (see Proposition \ref{prop:red number - reg}), which is of independent interest. It allows us to determine the reduction numbers in our Dedekind-Mertens-like formulas.  

There is an extensive literature on the Hilbert functions of determinantal rings (see, e.g, \cite{Abhy,  CnH, Gh, DnG, 
HT, KM, KP, KR, Kulkarni, R}. It often involves path counting arguments. Instead, we use a liaison-theoretic approach, based on the theory of Gorenstein liaison (see \cite{KMMNP, Mig}). 
\medskip 

{
We hope that our results motivate further investigations.  Thus, we conclude the article with some specific open questions outlined in Section~\ref{sec:Problems}. In particular, we discuss problems regarding   the shape  of minimal free resolutions, finding explicit minimal reductions, generalizations to higher minors, and  some connections to algebraic statistics. 
}



\section{Symmetric tableaux with holes: Gr\"obner bases}\label{sec:Gbases}
\label{sec:symmetric-ladders-G-bases}

In this section we determine Gr\"obner bases of a new class of  determinantal ideals, as mentioned in the introduction. We start by recalling our standard notation that is used throughout the paper.   The vector $\la = (\la_1,\ldots,\la_n)$ is a partition,  and $\mu = (\mu_1,\ldots,\mu_n)$ an integer vector such that 
\[
	0 \le \mu_1 \le \cdots \le \mu_n <  \la_n \le \cdots \le \la_1 =: m
\]
and $	\mu_i \ge i-1$  for all $i = 1,\ldots,n$. 
Entries of the tableaux  $\bT_{\la - \mu}$ correspond to variables $T_{i j}$ in the polynomial ring
\[
	K[\bT_{\la - \mu}] := K[T_{i j} \s 1 \le i \le n, \mu_i < j \le \la_i].
\]
Thinking of $\bT_{\la - \mu}$ as a subtableau of an $m \times n$ matrix,  the {\em symmetrized tableau} $\bS_{\la - \mu}$ is obtained by reflecting $\bT_{\la - \mu}$ along the main diagonal.  Note that the resulting symmetrization  may have holes along the main diagonal.

\begin{Example}
  \label{ex-symmetrizd tabeaux}
Consider $\la =(5, 5, 4)$ and $\mu = (1, 3, 3)$. Then we get

\begin{figure}[h!]
\includegraphics{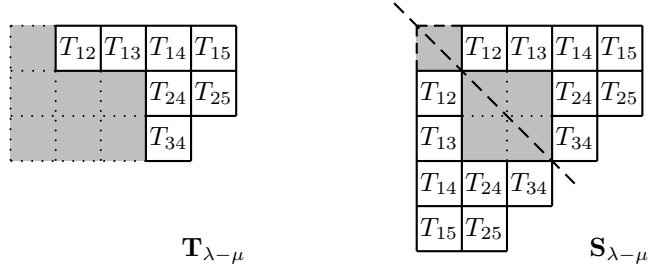}
\caption{Example of a symmetrized tableaux.}  
\end{figure}
\end{Example}

\newpage
Crucially, note also that in general   neither the  tableau $\bTlm$ nor $\bSlm$ is a ladder or a symmetric ladder, resp., in the usual sense (see, e.g., \cite{Conca} and \cite{G}).
\begin{Example}
  \label{ex:no-ladder}
Consider $\la =(5, 5, 4)$ and $\mu = (1,3,3)$. Then the variables $T_{1, 3}$ and $T_{2, 4}$ are in the tableau $\bTlm$. However, $T_{2, 3}$ is not in $\bTlm$ nor in  $\bSlm$, so the tableaux $\bTlm$ and  $\bSlm$ are not  ladders.
\end{Example}

Denote  by $I_2 (\bTlm)$ and $I_2 (\bSlm)$ the ideals in $K[\bTlm]$ generated by the determinants of $2 \times 2$ submatrices of $\bTlm$ and $\bSlm$, respectively.

\begin{Example}[Example \ref{ex-symmetrizd tabeaux}, continued]
  \label{ex-symmetrizd tabeaux contd}
For $\la$ and $\mu$ as in Example \ref{ex-symmetrizd tabeaux}, we get
\[
I_2 (\bTlm) = (T_{1 4} T_{2 5} - T_{1 5} T_{2 4})
\]
and
\[
I_2 (\bSlm) = (T_{1 4} T_{2 5} - T_{1 5} T_{2 4}, T_{12} T_{34} - T_{13} T_{24}).
\]
\end{Example}

The main result of this section is a Gr\"obner basis computation (see Theorem \ref{thm:Minors-Grob}).  To this end, we fix throughout this section the lexicographic order $\<$  on the monomials in $K[\bTlm]$, where the variables are ordered row-wise, that is, $\<$ is the lexicographic order induced by
\[
    T_{1,\mu_1+1} > T_{1,\mu_1+2}>\dots >T_{1,\lambda_1} >
    T_{2,\mu_2+1} > \dots > T_{2,\lambda_2} >
    \dots > T_{n,\lambda_n}.
\]
Equivalently, $T_{rc}> T_{r'c'}$ if $r<r'$, or if $r=r'$ and $c<c'$. Note that this is a diagonal term order, that is, the leading term of any minor is the main diagonal term.

\begin{Theorem}
    \label{thm:Minors-Grob}
\begin{itemize}
  \item[(a)] The 2-minors of $\bTlm$ form a Gr\"obner
  basis of  $I_2(\bTlm)$ with respect to the lexicographic order $\<$.

  \item[(b)] The 2-minors of $\bSlm$ form a Gr\"obner
  basis of  $I_2(\bSlm)$ with respect to the lexicographic order $\<$.
\end{itemize}
\end{Theorem}

\begin{proof}
We first show (b).  We use induction on the number of rows, $n$, of $\bTlm$. If $n =1$, then $I_2 (\bSlm) = 0$, so the claim is clearly true. Let $n \ge 2$. Then we consider the partition $\widetilde{\la}$ differing from $\la$ only in its
last part:
\[
\widetilde{\la} := (\la_1,\ldots,\la_{n-1}, \la_n - 1).
\]
Thus, the tableaux  $\bTlm$  is obtained from ${\bTwlm}$ by adding
a new right-most box in its last row. Using induction on the number of variables in row $n$ of $\bTlm$ we may assume that the 2-minors of $\bSwlm$ form a Gr\"obner basis of $I_2 (\bSwlm)$ with respect to the lexicographic order. To show the analogous claim for the 2-minors of $\bSlm$ we simply use Buchberger's Criterion (see, e.g.,  \cite{CLO})     and show that the S-polynomial of any two minors has remainder zero after at most four steps of the division algorithm.

Let $M_1$ and $M_2$ be two distinct $2$-minors of the symmetric tableau $\bSlm$.    To simplify notation, throughout much of this proof, let us use a single index to denote the row and column indices for the variables in    $\bSlm$.
Let
    \begin{align*}
        M_1 :=& T_dT_b-T_aT_c, \\
        M_2 :=& T_lT_k-T_eT_f,
    \end{align*}
    where the positive term in each binomial represents the initial term of the minor with respect to the order $\<$.

We may assume that the initial terms of $M_1$ and $M_2$ are not relatively prime,
    since their S-polynomial reduces to zero otherwise (see, for example, \cite[Proposition 2.9.4]{CLO}).
    In addition, if the leading terms are not relatively prime, say $T_l=T_d$,
    we may assume that the trailing terms \emph{are} relatively prime.  Indeed, if for example $T_a=T_e$,
    the S-polynomial will be a multiple of another quadric in the ideal (another $2$-minor of $\bSlm$), and will reduce to zero:
    \[
        S(M_1,M_2) = T_e(T_bT_f - T_kT_c) .
    \]
Thus we may assume that $\{T_e, T_f\} \cap \{T_a,T_c\} = \emptyset$. Since $M_1$ and $M_2$ are minors, it follows that either the support of $M_1 M_2$ consists of exactly 7 variables or it consists of 6 variables and the leading monomials of $M_1$ and $M_2$ are equal. Furthermore, by symmetry of $\bSlm$, we may assume that all the 2-minors we are considering are determinants of matrices whose south-east corners are not in the lower half of $\bSlm$. It follows that any such minor with $T_{n, \la_n}$ in its support is the determinant of a matrix with $T_{n, \la_n}$ in its south-east corner, where $T_{n, \la_n}$ is located in $\bTlm$. Moreover, if $T_{n, \la_n}$ does not divide the leading term of such a minor, then the minor is in $I_2 (\bSwlm)$. In any case, the row indices of the matrices determining our minors are at most $n$.

We now treat separately the possibilities for the variable $T_{n, \la_n}$ to appear in the support of one, both, or none of the two minors.
\smallskip

\paragraph{\bf \emph{Case I}} Suppose that $T_{n, \la_n}\not\in supp(M_1) \cup supp(M_2)$. Then $M_1, M_2 \in I_2 (\bSwlm)$. By induction hypothesis, the S-polynomial $S(M_1, M_2)$ can be reduced to zero using 2-minors in $I_2 (\bSwlm)$. Thus, this is true in $I_2 (\bSlm)$ as well.
\smallskip

\paragraph{\bf\emph{Case II}} Suppose that $T_{n, \la_n}\in supp(M_1)$, but $T_{n, \la_n} \not\in  supp(M_2)$, say, $T_{n, \la_n}=T_d$.
By the leading term criterion, we may assume that $T_b$ appears in the leading monomial of $M_2$. Letting $l=b$ provides
    \begin{align*}
        M_1 &= T_{n, \la_n} T_b-T_a T_c , \\
        M_2 &= T_b T_k - T_e T_f .
    \end{align*}
Then $S(M_1,M_2) = T_aT_cT_k-T_eT_fT_{n, \la_n}$, and $T_b$ is located to the left and above of $T_{n, \la_n}$. Since by our convention on the south-east corners of minors $T_k$ is not in a row with index greater than $n$, the variable $T_k$ also must be located above row $n$. Thus, schematically, there are the following possibilities for the relative positions of the variables in the supports of $M_1$ and $M_2$:
            \[  \begin{matrix}
                \ul T_k & T_f & \textcolor{dgreen}{T_{m'}}\\
                T_e & \textcolor{black}{\ul{ \ul T_b}} & T_a\\
                 & T_c & {\ul T_{n, \la_n}}
                \end{matrix}
                \quad  \textcolor{black}{\text{or }} \quad
                \begin{matrix}
                \textcolor{black}{\ul{ \ul T_b}} & T_a & T_e\\
                T_f & \textcolor{black}{T_{m'}} & {\ul T_k}\\
                T_c & {\ul T_{n, \la_n}}
                \end{matrix}
                \quad  \textcolor{black}{\text{or }} \quad
                \begin{matrix}
                \textcolor{black}{\ul{ \ul T_b}} & T_e  &  T_a\\  
                T_f & \ul T_k &  \textcolor{black}{T_{m'}}\\  
                 T_c  &  &   {\ul T_{n, \la_n}}    
                \end{matrix}.
            \]
The variables in each initial term are underlined, making the common one underlined twice. Furthermore, $T_{m'}$ denotes a variable that must be present in $\bTlm$ because $T_{n, \la_n}$ is. It will be used for reduction. Indeed,  in all cases we can reduce the S-polynomial to zero because
            \[
                S(M_1,M_2)   = T_c(T_aT_k-T_eT_{m'}) + T_e(T_cT_{m'}-T_fT_{n, \la_n}).
            \]
Notice that the order of the two steps in the division algorithm depends on the leading term of $S(M_1, M_2)$. The indicated reduction works in all cases.
\smallskip

\paragraph{\bf\emph{Case III}} Finally, suppose $T_{n, \la_n} = supp(M_1)\cap supp(M_2)$, say,
    $T_{n, \la_n} = T_d = T_l$, and the support of $M_1 M_2$ consists of 7 variables.  Then
    \begin{align*}
        M_1 &= T_{n, \la_n}T_b-T_aT_c , \\
        M_2 &= T_{n, \la_n}T_k-T_eT_f,
    \end{align*}
    where $T_k \neq T_b$, and
    \[
        S(M_1,M_2) = T_kT_aT_c - T_bT_eT_f.
    \]
The variables $T_b$ and $T_k$ must be located to the left and above of $T_{n, \la_n}$.    One typical situation  for the positions of the involved variables is:
\\
            \[ \begin{matrix}
              \ul T_k & \textcolor{dgreen}{T_{n'}} & T_e\\
              \textcolor{dgreen}{T_{m'}} & \ul T_b & T_a\\
              T_f & T_c & \textcolor{black}{\ul{ \ul T_{n, \la_n}}}
                \end{matrix} .
            \]
As before, $T_{m'}$ and $T_{n'}$ denote variables whose presence is established if it is needed in the reduction process.

Indeed, assume $T_k$ is not located in the lower half of $\bSlm$. Then $T_{n'}$ is in $\bTlm$. Thus, the division algorithm provides
\[
S(M_1, M_2) - T_a (T_k T_c - T_f T_{n'}) = T_f (T_a T_{n'} - T_b T_e).
\]
Otherwise, if $T_k$ is in the lower half of $\bSlm$, then $T_{m'}$ must be present there as well. This time the division algorithm gives
\[
S(M_1, M_2) - T_c (T_k T_a - T_e T_{m'}) = T_e (T_c T_{m'} - T_b T_f).
\]
Hence in both cases the S-polynomial reduces to zero.

The other typical situation is:
            \\
            \[ \begin{matrix}
                 & \ul T_b & T_a\\
                 \ul T_k & \textcolor{dgreen}{T_{m'}} & T_e\\
                 T_f & T_c & \textcolor{black}{\ul{ \ul T_{n, \la_n}}}
                \end{matrix} .
            \]
Assume first that the variable $T_{m'}$ is present in $\bSlm$.
There are two cases. If the leading term of the S-polynomial is $T_kT_aT_c$, then we use  the minor $T_kT_c-T_{m'}T_f$ whose leading term divides $T_kT_aT_c$. Thus the division algorithm provides
\[S(M_1,M_2) - T_a(T_kT_c-T_{m'}T_f) = T_f (T_aT_{m'}-T_bT_e),\]
and hence the S-polynomial reduces to zero as a multiple of another minor.
Otherwise, if the leading term of the S-polynomial is $T_bT_eT_f$, then it is divisible by the leading term of $T_bT_e-T_aT_{m'}$. Thus we can again reduce the S-polynomial to zero using the division algorithm:
\[ S(M_1,M_2) + T_f(T_bT_e-T_aT_{m'}) = T_a(T_kT_c-T_fT_{m'}).
\]

It remains to consider the case where the variable $T_{m'}$ is not present in the tableau $\bSlm$. It follows that $T_b$ and $T_e$ must be located in the upper half of $\bSlm$,  whereas $T_k$ and $T_c$ are in the lower half of the tableau. In particular,  none of these variables is on its main diagonal. We need to keep track of the positions of the involved variables. Denote the rows and columns of the locations of these variables by $i, j, n$ and $p, q, \la_n$, respectively. Thus, 
            \begin{equation}
              \label{eq:ineq1}
              i < j < n \le \la_n, \quad p < q < \la_n, \quad p < j, \text{ and }  i < q < n.
            \end{equation}
            Returning to the original double indices for the variables, the above diagram becomes
            \[ \begin{matrix}
              & p & q & \la_n \\[2pt]
             i && T_{i, q} & T_{i, \la_n}\\
             j & T_{p, j} & & T_{j, \la_n}\\
             n & T_{p. n} & T_{q, n} & T_{n, \la_n}
            \end{matrix}.
            \]
            Here we included the row and column indices and wrote the variables in the form $T_{k, l}$ with $k \le l$. Notice that the S-polynomial of the minors $M_1$ and $M_2$ now reads as
            \begin{equation*}
              S(M_1, M_2) = T_{i, q} T_{j, \la_n} T_{p, n} - T_{i, \la_n} T_{p, j} T_{q, n}.
            \end{equation*}
            The non-presence of the variable $T_{m'}$ means that $j \le q \le \mu_j$ or $q \le j \le \mu_q$.

            \paragraph{\bf\emph{Case A}} Assume $j \le q \le \mu_j$. Now we consider two subcases by comparing $i$ and $p$.

            \paragraph{\bf\emph{Case A.1}} Assume $i < p$. Then  the following relations hold: 
            \begin{equation*}
              \label{eq:ineq2}
              i <  p  < j \le q < n \le \la_n.
            \end{equation*}
            Hence the leading monomial of $S(M_1, M_2)$ is $T_{i, q} T_{j, \la_n} T_{p, n}$. Using rows $i, p$ and columns $q, n$ we claim that
            \[
            T_{i,q} T_{p, n} - T_{i, n} T_{p, q} \in I_2 (\bSlm).
            \]
            Indeed, since $i < n$ and $T_{p, n}$ is present in $\bTlm$, its column $n$ also contains $T_{i, n}$. Moreover, the presence of $T_{p, j}$ means that $\mu_p < j$. Since $j \le q < n \le \la_p$, we conclude that $T_{p, q}$ is in $\bTlm$. This shows the existence of the above minor. Its leading monomial is $T_{i,q} T_{p, n}$. Hence we can use it in the division algorithm for reducing $S(M_1, M_2)$. We obtain
            \[
            S(M_1, M_2) - T_{j, \la_n} (T_{i,q} T_{p, n} - T_{i, n} T_{p, q}) = T_{i, n} T_{p, q} T_{j, \la_n}  - T_{i, \la_n} T_{p, j} T_{q, n} =: F.
            \]
            The leading monomial of $F$ is $T_{i, n} T_{p, q} T_{j, \la_n}$. Now, using rows $i, j$ and columns $n, \la_n$, we claim
            \[
            T_{i, n} T_{j, \la_n} - T_{i, \la_n} T_{j, n} \in I_2 (\bSlm).
            \]
            To this end it is enough to see that the variable $T_{j, n}$ is present in $\bSlm$ if $n < \la_n$. However, $T_{q, n} \in \bTlm$ implies $\mu_q < n$. Hence, using $j \le q$, we conclude that $\mu_j \le \mu_q < n \le \la_q \le \la_j$ which gives $T_{j, n} \in \bTlm$. If $n < \la_n$, then the leading monomial of the last minor is $T_{i, n} T_{j, \la_n}$. Thus, we can use it in another step of the division algorithm. We get
            \[
            F - T_{p, q} (T_{i, n} T_{j, \la_n} - T_{i, \la_n} T_{j, n}) = T_{i, \la_n} (T_{p, q} T_{j, n} - T_{p, j} T_{q, n}).
            \]
            Notice that this is also true if $n = \la_n$.  Using, rows $j, q$ and columns $p, n$ we see that $T_{p, q} T_{j, n} - T_{p, j} T_{q, n}$ is a minor of $\bSlm$ or trivial if $j = q$. In both cases, $S(M_1, M_2)$ reduces to zero.

            \paragraph{\bf\emph{Case A.2}} Assume $i \ge p$. Then  the following relations hold: 1
            
            \[
            p \le i < j \le q < n.
            \]
            It follows that the leading monomial of $S(M_1, M_2)$ is $T_{p, j} T_{i, \la_n} T_{q, n}$. Using rows $j, q$ and columns $p, n$ we claim that
            \[
            T_{p, j} T_{q, n} - T_{j, n} T_{p, q} \in I_2 (\bSlm).
            \]
            Indeed, since $q < n$ and $T_{q, n}$ is in $\bTlm$, its column $n$ also contains $T_{j, n}$. As above, the presence of $T_{p, j}$ means that $\mu_p < j$. Since $j \le q \le \la_j$, we get $T_{p, q} \in \bTlm$, as desired. The leading monomial of the above minor is $T_{p, j} T_{q, n}$. Thus, the division algorithm provides
            \[
            S(M_1, M_2) + T_{i, \la_n} (T_{p, j} T_{q, n} - T_{j, n} T_{p, q}) = - T_{p, q} T_{i, \la_n} T_{j, n}  + T_{p, n} T_{i, q} T_{j, \la_n}  =: F.
            \]
            The leading monomial of $F$ is $T_{p, q} T_{i, \la_n} T_{j, n}$. Using rows $p, i$ and columns $q, \la_n$ we claim that
            \[
            T_{p, q} T_{i, \la_n} - T_{p, \la_n} T_{i, q} \in I_2 (\bSlm).
            \]
            Indeed, we have already seen $T_{p, q} \in \bTlm$. Moreover, since
            $T_{i, \la_n}$   is in column $\la_n$ and $p \le i$, this column also
            contains $T_{p, \la_n}$. The leading monomial of the last minor is $T_{p, q} T_{i, \la_n}$. Hence, another step in the division algorithm gives
            \[
            F + T_{j, n} (T_{p, q} T_{i, \la_n} - T_{p, \la_n} T_{i, q}) = T_{i, q} (T_{p, n} T_{j, \la_n} - T_{p, \la_n} T_{j, n}).
            \]
            Since this is trivial or a multiple of a minor of $\bSlm$ using rows $p, j$ and columns $n, \la_n$,  $(M_1, M_2)$ has been reduced to zero, as desired.

             \paragraph{\bf\emph{Case B}} Assume $q \le j \le \mu_q$. Again we consider two subcases by comparing $i$ and $p$.

            \paragraph{\bf\emph{Case B.1}} Assume $i \le p$. This implies  the relations
            \[
            i \le p < q \le j < n \le \la_n.
            \]
            Thus, the leading monomial of $S(M_1, M_2)$ is $T_{i, q} T_{p, n} T_{j, \la_n}$. Using rows $q, j$ and columns $i, \la_n$ we obtain
            \[
             T_{i, q} T_{j, \la_n} - T_{i, j} T_{q, \la_n}  \in I_2 (\bSlm).
            \]
            Indeed, since $q \le j$ and $T_{j, \la_n}$ is present in $\bTlm$, its column $n$ also contains $T_{q, \la_n}$. Moreover, the presence of $T_{p, j}$ means $\mu_p < j$. Hence, we get  $\mu_i \le \mu_p <j < n \le \la_n \le \la_i$, so $T_{i, j} \in \bTlm$. The leading term of the above minor is $T_{i, q} T_{j, \la_n}$. Applying the division algorithm, we obtain
            \[
            S(M_1, M_2) - T_{p, n} ( T_{i, q} T_{j, \la_n} - T_{i, j} T_{q, \la_n}) =  T_{i, j} T_{p, n}  T_{q, \la_n} - T_{i, \la_n} T_{p, j} T_{q, n} =: F
            \]
            with leading monomial $T_{i, j} T_{p, n}  T_{q, \la_n}$. Using rows $i, p$ and columns $j, n$ we claim that
            \[
            T_{i, j} T_{p, n} - T_{i, n} T_{p, j} \in I_2 (\bSlm).
            \]
            Indeed, column $n$ of $\bTlm$ contains $T_{p, n}$. Since $i \le p$, the variable $T_{i, n}$ is also in this row. Observe that the leading term of this minor is $T_{i, j} T_{p, n}$. Using the minor for another step of the division algorithm we obtain
            \[
            F - T_{q, \la_n} (T_{i, j} T_{p, n} - T_{i, n} T_{p, j}) = T_{p, j} (T_{i, n} T_{q, \la_n} - T_{i, \la_n} T_{q, n}).
            \]
            This polynomial is trivial or a minor of $\bSlm$ using rows $i, q$ and columns $n, \la_n$. Hence it reduces to zero.

            \paragraph{\bf\emph{Case B.2}} Assume $i > p$. Then the following relations hold: 
            \[
            p < i \le q \le j < n.
            \]
            Thus, the leading term of $S(M_1, M_2)$ is $T_{p, j} T_{i, \la_n} T_{q, n}$. Using rows $p, i$ and columns $j, \la_n$ we get
            \[
            T_{p, j} T_{i, \la_n} - T_{p, \la_n} T_{i, j} \in I_2 (\bSlm).
            \]
            Indeed, $T_{p, \la_n}$ is in column $n$ of $\bTlm$ because $T_{i, \la_n}$ is and $p < i$. Furthermore, the presence of $T_{i, q}$ implies $\mu_i < q \le j < n \le \la_i$, and thus $T_{i, j} \in \bTlm$. Notice that the leading monomial of the last minor is $T_{p, j} T_{i, \la_n}$. Now the division algorithm gives
            \[
            S(M_1, M_2) + T_{q, n} (T_{p, j} T_{i, \la_n} - T_{p, \la_n} T_{i, j}) = T_{p, n} T_{i, q} T_{j, \la_n}  - T_{p, \la_n} T_{i, j} T_{q, n} =: F \]
            whose leading monomial is $T_{p, n} T_{i, q} T_{j, \la_n}$. Using rows $q, j$ and columns $i, \la_n$, we claim
            \[
            T_{i, q} T_{j, \la_n} - T_{i, j} T_{q, \la_n} \in I_2 (\bSlm).
            \]
            To this end it suffices to notice that $T_{q, \la_n}$ is present in column $\la_n$ of $\bTlm$ because $T_{j, \la_n}$ is and $q \le j$. Since $T_{i, q} T_{j, \la_n}$ is the leading monomial of this minor we can use it in the division algorithm. We get
            \[
            F - T_{p, n} (T_{i, q} T_{j, \la_n} - T_{i, j} T_{q, \la_n}) = T_{i, j} (T_{p, n} T_{q, \la_n} - T_{p, \la_n} T_{q, n}).
            \]
            Again, this is zero or a minor of $\bTlm$ using rows $p, q$ and columns $n, \la_n$. Hence  $S(M_1, M_2)$ has been reduced to zero.
\smallskip

\paragraph{\bf\emph{Case IV}} Suppose the leading monomials of $M_1$ and $M_2$ are equal and divisible by $T_{n, \la_n}$. Thus, the support of $M_1 M_2$ consists of 6 variables. In order to keep track of locations we use again double indices. Write the leading monomial of $M_1$ as $T_{i, j} T_{n, \la_n}$.  Since $M_1 \neq M_2$,we must have $i \neq j$, say $i < j$. Thus, we may assume
\begin{align*}
  M_1 = & T_{i, j} T_{n, \la_n} - T_{i, \la_n} T_{j, n}\\
  M_2 = & T_{i, j} T_{n, \la_n} -  T_{i, n} T_{j, \la_n},
\end{align*}
where $n < \la_n$. Hence
\[
S(M_1, M_2) = T_{i, n} T_{j, \la_n} - T_{i, \la_n} T_{j, n},
\]
which is a minor of $\bSlm$ using rows $i, j$ and columns $n, \la_n$.
This completes the proof of (b).
\smallskip

Finally, consider Claim (a). It also follows from the above arguments, but its proof is simpler because the second situation in Case III above does not occur. We omit the details. 
\end{proof}


\section{Symmetric tableaux with holes: Invariants}\label{sec:invariants}
\label{sec:symmetric-ladders-invariants}

Theorem \ref{thm:Minors-Grob} allows us to compute the initial ideals of the ideals
$I_2(\bTlm)$ and $I_2(\bSlm)$ with respect to the order $\<$. We use these to determine invariants of the  determinantal ideals themselves.
In order to analyze their properties we use a technique from liaison
theory.

\begin{Proposition}
  \label{prop:bdl}
Let $J \subset I \subset R = K[x_1,\ldots,x_r]$ be homogeneous
ideals such that $\hht J = \hht I - 1$. Let $f \in R$ be a
homogeneous polynomial of degree $d$ such that $J  : f = J$, and set $I' := f
\cdot I + J$. If $R/I$ and $R/J$ are Cohen-Macaulay, then so is
$R/I'$ and $\hht I' = \hht I$.

Moreover, the Hilbert functions of the involved rings are related by
\begin{equation*}
  \label{eq:hilb-liaison}
h_{R/I'} (j) = h_{R/I} (j-d) + h_{R/J} (j) - h_{R/J} (j-d) \quad \text{for all } j \in \ZZ.
\end{equation*}
\end{Proposition}

\begin{proof}
This is part of Lemma 4.8 in \cite{KMMNP}.
\end{proof}

\begin{Remark}
  \label{rem:bdl}
(i) The ideal $I' := f \cdot I + J$ is called a {\em basic double
link} of $I$. The name stems from the fact that $I'$ can be
Gorenstein linked to $I$ in two steps if $I$ is unmixed and $R/J$ is
Cohen-Macaulay and generically Gorenstein (\cite[Proposition
5.10]{KMMNP}).

(ii) A homogenous ideal $I$ is said to be {\em glicci} if it is in
the Gorenstein liaison class of a complete intersection. It then
follows that $I$ is Cohen-Macaulay. If $I$ is a squarefree monomial
ideal, then, following \cite[Definition 2.2]{NR}, $I$ is said to be
{\em squarefree glicci} if $I$ can be linked in an even number of
steps to a complete intersection $I'$  generated by variables such
that every other ideal in the chain linking $I$ to $I'$ is a
squarefree monomial ideal.

Note that Proposition \ref{prop:bdl} provides: If $I$ is a
squarefree monomial ideal that can be obtained from an ideal
generated by variables by a sequence of basic double links, then $I$
is squarefree glicci, thus in particular Cohen-Macaulay.
\end{Remark}

We use basic double links to show the initial ideals we consider correspond to simplicial complexes that satisfy a strong combinatorial property: they are vertex decomposable. Recall that a \emph{simplicial complex} $\Delta$ on $n$ vertices is a collection of subsets of $\{1,\ldots,n\}$ that is closed under inclusion. The elements of $\Delta$ are called the \emph{faces} of $\Delta$. The dimension of a face $F$ is $|F| - 1$, and the \emph{dimension} of $\Delta$ is the maximum dimension of its faces. The complex $\Delta$ is said to be \emph{pure} if all its \emph{facets}, the faces that are maximal with respect to inclusion, have the same dimension.

Let $\{k\}$ be a vertex of $\Delta$, a 0-dimensional face. Then the \emph{link} of $k$ is
\[
\lk_{\Delta} (k) = \{G \in \Delta \mid \{k\} \cup G\in\Delta, \{k\} \cap
G=\emptyset\},
\]
and the \emph{deletion} with respect to $k$ is
\[
\Delta_{-k}=\{G\in\Delta\mid \{k\}\cap G=\emptyset\}.
\]
A simplicial complex $\Delta$ is \emph{vertex decomposable} if it is a
simplex, or it is the empty set, or there exists a vertex $k$ such that
$\lk_k(\Delta)$ and $\Delta_{-k}$ are both pure and vertex decomposable,
and
\[
\dim\Delta=\dim(\Delta_{-k})=\dim\lk_{\Delta} (k)+1.
\]
Vertex decomposable simplicial complexes are known to have strong structural properties. In particular, they are shellable, and thus Cohen-Macaulay.

The \emph{Stanley-Reisner ideal} associated to a simplicial complex $\Delta$ on
$n$ vertices is the squarefree monomial ideal
\[
I_{\Delta}=(x_{i_1},\ldots,x_{i_s}\mid
\{i_1,\ldots,i_s\}\not\in\Delta)\subset K[x_1,\ldots,x_n].
\]
In fact, this induces a bijection between the simplicial complexes on $n$ vertices and squarefree monomial ideals in $K[x_1,\ldots,x_n]$.
According to \cite[Theorem 3.3]{NR}, the Stanley-Reisner ideal of a vertex decomposable simplicial complex is squarefree glicci. In the main result of this section we show first that the ideals in question are squarefree glicci by describing explicitly the required basic double links, and then use this to infer the desired vertex decomposability.

\begin{Theorem}
  \label{thm:initial-ideals}
\begin{itemize}
  \item[(a)] The initial ideal
  $\ini (I_2(\bTlm)) : = \ini_{\<} (I_2(\bTlm))$ is
  squarefree and has height
\[
\hht \ini (I_2(\bTlm)) = \sum_{i=2}^n (\la_i - \mu_i - 1).
\]
Its associated simplicial complex is vertex decomposable. In particular, $\ini (I_2(\bTlm))$ is Cohen-Macaulay.

  \item[(b)]  The initial ideal
  $\ini (I_2(\bSlm)) : = \ini_{\<} (I_2(\bSlm))$ is
  squarefree and has height
\[
\hht \ini (I_2(\bSlm)) = \max \{0, n-1-\mu_1\} +  \sum_{i=2}^n (\la_i - \mu_i - 1).
\]
Its associated simplicial complex is vertex decomposable. In particular, $\ini (I_2(\bSlm))$ is Cohen-Macaulay.

\end{itemize}
Moreover, if either ideal is non-trivial, then it can be
obtained from an ideal generated by variables using suitable basic
double links. In particular, it is squarefree glicci.
\end{Theorem}

\begin{proof}
In both cases we use induction on the number $n$ of rows of
$\bTlm$. If $n =1$, then $I_2(\bTlm)$ and $I_2(\bSlm)$ are trivial
and there is nothing to show.

Let $n \ge 2$. Now we use induction on $\la_n - \mu_n \ge 1$. We define a
new partition $\widetilde{\la}$ differing from $\la$ only in its
last part by
\[
\widetilde{\la} := (\la_1,\ldots,\la_{n-1}, \la_n - 1).
\]
Thus, the tableaux ${\bTwlm}$ is obtained from $\bTlm$ by deleting
the right-most box in its last row. It follows that
\[
\ini (I_2(\bTwlm))  \subset \ini (I_2(\bTlm))
\]
and
\[
\ini (I_2(\bSwlm))  \subset \ini (I_2(\bSlm)).
\]
We first determine how much larger the ideals on the right-hand side
are. We  treat the two cases separately.
\smallskip

(a)  Observe if
$\la_n - \mu_n = 1$, that is, the last row of $\bTlm$ consists of
precisely one box, then deleting this box gives a tableaux leading
to the same ideals as the given ones. Thus, we conclude by induction
on the number of rows. Now assume $\la_n - \mu_n \ge 2$. 

Theorem \ref{thm:Minors-Grob}(a) provides that
\begin{equation}
  \label{eq:bdl-ladder prep}
 \ini (I_2(\bTlm)) = T_{n, \la_n} \fa + \ini (I_2(\bTwlm)),
\end{equation}
where
\[
\fa = (T_{i j} \s 1 \le i < n \text{ and } \mu_n < j < \la_n).
\]
Using induction, it follows that $\ini (I_2(\bTlm))$ is a squarefree monomial ideal.

Now note that we can rewrite Equation \eqref{eq:bdl-ladder prep} as
\begin{equation}
  \label{eq:bdl-ladder}
 \ini (I_2(\bTlm)) = T_{n, \la_n} \fa' + \ini (I_2(\bTwlm)),
\end{equation}
where
\begin{equation}
  \label{eq:linked-ideal-ladder}
\fa' = \fa + \ini (I_2(\bTwlm)).
\end{equation}
Set now
\[
\la' := (\la_1 - (\la_n - \mu_n -1),\ldots,\la_{n-1} - (\la_n - \mu_n -1)) \in \ZZ^{n-1}
\]
and
\[
\mu' := (\mu_1,\ldots,\mu_{n-1}) \in \ZZ^{n-1}.
\]
Then the ideal $\fa'$ is isomorphic to (the extension ideal  in
$K[\bTlm]$) of the sum of $\ini (I_2 (\bT_{\la' - \mu'}))$ and an
ideal generated by $\hht \fa$ new variables. Applying the induction
hypothesis to $\ini (I_2 (\bT_{\la' - \mu'}))$ we conclude that
$\fa'$ is a Cohen-Macaulay ideal of height
\begin{eqnarray*}
\hht \fa' & = & \hht \fa + \hht \ini (I_2 (\bT_{\la' - \mu'})) \\
          & = & (n-1) (\la_n - \mu_n -1) + \sum_{i=2}^{n-1} (\la_i - (\la_n - \mu_n - 1) - \mu_i - 1) \\
          & = & \sum_{i = 2}^{n} (\la_i - \mu_i -1) \\
          & = & 1 + \hht \ini (I_2 (\bTwlm)).
\end{eqnarray*}
Hence $\ini (I_2 (\bTlm))$ is a basic double link of $\fa'$, and   Proposition \ref{prop:bdl} shows that $\ini (I_2(\bTlm)) $ has the claimed height.

Denote by $\Delta$ the simplicial complex corresponding to $\ini (I_2(\bTlm))$.
By induction hypothesis, the simplicial complex of $\ini (I_2(\bTwlm))$ is vertex decomposable. Thus, this is also true for the simplicial complex corresponding to the ideal $\fa'$. Hence Equation \eqref{eq:bdl-ladder} shows that the link $\lk_{\Delta} (n,\la_n)$ and the deletion $\Delta_{-(n, \la_n)}$  with respect to the vertex $(n,\la_n) \in \Delta$ are vertex decomposable, and hence so is $\Delta$.
\smallskip

(b) We employ the same strategy as for (a), though carrying it out is more involved. Theorem \ref{thm:Minors-Grob}(b) implies that
\begin{equation}
  \label{eq:pre-symm-bdl}
\ini (I_2(\bSlm)) = T_{n, \la_n} \fb + \ini (I_2(\bSwlm)),
\end{equation}
where $\fb$ is an ideal that is generated by variables. We now determine
this ideal $\fb$. To this end, one must  list all the 2-minors
of $\bSlm$ such that $T_{n, \la_n}$ is an entry on the main
diagonal. By symmetry of $\bSlm$ we may assume that $T_{n, \la_n}$
is in $\bTlm$. Thus, we are looking for $2 \times 2$ submatrices of
$\bSlm$ that are formed by rows $i$ and $n$, where $i < n$, and
columns $j$ and $\la_n$ of $\bSlm$, where $j < \la_n$. We
distinguish three cases.
\smallskip

{\it Case 1}: Assume position $(n, j)$ is in $\bTlm$. This is true
if  and only if $T_{i j}$ is in the ideal $\fa$.

{\it Case 2}: Assume that position $(i, j)$ is in $\bTlm$, but that
position $(n, j)$ is not in $\bTlm$. The first condition means
$\mu_i < j < \la_n$, whereas the second condition gives $j \le
\mu_n$. Furthermore,  since position $(n, j)$ is in $\bSlm$, by
symmetry the second condition implies that position $(j, n)$ is in
$\bTlm$, that is, $j \le n$ and $\mu_j < n$. Notice that the
condition $n < \la_j$ is always satisfied because our assumptions
provide $n-1 \le \mu_n < \la _n \le \la_j$. Since, by assumption
$n-1 \le \mu_n < \la_n$ and  the condition $j-1 \le \mu_j < n$
implies $j \le n$, we see that Case 2 occurs if and only if $\mu_i <
j \le \min \{n, \mu_n\}$ and $\mu_j < n$.

{\it Case 3}: Assume that positions $(i, j)$ and $(n, j)$ are  not
in $\bTlm$, that is, $j \le \mu_i$ and the positions $(j, i)$ and
$(j, n)$ are in $\bTlm$.  The latter is equivalent to $j \le n$ and
$\mu_j < i < n$, using again that $n < \la_j$. Notice that here we
have $i \ge j$, thus the variable at position $(i, j)$ is $T_{j i}$.
\smallskip

These considerations show that we can write
\[
\fb = \fa + \fb_1 + \fb_2,
\]
where the ideal
\[
\fb_1 = (T_{i j } \s \mu_j < n \text{ and } \mu_i <
j \le \min \{n, \mu_n\})
\]
corresponds to Case 2 and
\[
\fb_2 = (T_{i j } \s i \le \mu_j \text{ and }
 \mu_i < j < n)
\]
corresponds to Case 3. Using induction and Equation \eqref{eq:pre-symm-bdl}, it follows that $\ini (I_2 (\bSlm))$ is a squarefree monomial ideal.

Next, we claim that $\fb_2$ can be rewritten as
\begin{equation}
  \label{eq:b2}
\fb_2  = (T_{i j} \s \max \{i, \mu_i \} < j < n \}.
\end{equation}

Indeed, the right-hand side is contained in $\fb_2$ because $i < j$
implies $i \le j-1 \le \mu_j$.

Conversely, if $T_{i j} \in \fb_2$, then $i-1 \le \mu_i < j$. Assume
$i = j$. Then, we get $\mu_j = \mu_i = i-1$, a contradiction to $i <
\mu_j$. Thus, $T_{i j}$ is in the right-hand side of Equation
\eqref{eq:b2}, which establishes said equation.

Now we are ready to rewrite the ideal $\fb$ as
\begin{equation}
  \label{eq:min-gen-b}
\fb = \fa + \fb_1 + \fb_2 =  \fa + \fb'',
\end{equation}
where
\[
\fb'' := (T_{i j} \s 1 \le i < n \text{ and } \mu_i < j < n).
\]

Indeed, clearly $\fb_2$ is contained in $\fb''$. Assume  there is
some $T_{i j} \in \fb_1 \setminus \fb''$. This provides $n = j \le
\mu_n$, a contradiction to $\mu_j < n$.

Conversely, assume there is some $T_{i j} \in \fb'' \setminus \fb$.This condition together with 
 $\fb_2 \subset \fb$ forces $i= j > \mu_i \ge i-1$. 
This implies $\mu_j = j-1 < n$ and $\mu_i = i-1 < j \le n-1 \le
\mu_n$. Thus $T_{i j}$ is in $\fb_1 \subset \fb$, a contradiction.

Using Equation \eqref{eq:min-gen-b} we conclude that
\[
\ini (I_2(\bSlm)) = \left (T_{i j} T_{k l} \; \bigg | \; \begin{array}{l}
1 \le i < k \le n, \; 1 \le j < l \le \la_k, \\  \text{and }  
\mu_k < j \text{ or } \mu_i < j < k
\end{array}
\right ).
\]
It follows that we can rewrite Equation \eqref{eq:pre-symm-bdl} as
\begin{equation}
  \label{eq:symm-bdl}
\ini (I_2(\bSlm)) = T_{n, \la_n} \fb' + \ini (I_2(\bSwlm)),
\end{equation}
where
\begin{eqnarray*}
\fb' & := & \fb + \ini (I_2(\bSwlm)) \\
& = & \fb + \left (T_{i j} T_{k l} \; \bigg | \; \begin{array}{l}
1 \le i < k < n, \; 1 \le j < l \le \la_k,  \text{ and }  \\
\mu_k < j \text{ or } \mu_i < j < k,  \text{ and } \\
n \le j < l \le \mu_n \text{ or } \la_n \le j \text{ or }
    [n \le j \le \mu_n \text{ and } \la_n \le l]

\end{array}
\right ).
\end{eqnarray*}
Set now
\begin{equation*}
\la' := (\la_1 - (\la_n - \mu_n -1),\ldots,\la_{n-1} - (\la_n - \mu_n -1)) \in \ZZ^{n-1}
\end{equation*}
and
\begin{equation*}
\mu' := (\mu_1',\ldots,\mu_{n-1}') \in \ZZ^{n-1}, \; \text{ where }
\mu_i' := \max \{n-1, \mu_i\}.
\end{equation*}
Then it follows that the ideal $\fb'$ is isomorphic to (the
extension ideal in $K[\bTlm]$) of the sum of $\ini (I_2(\bT_{\la' -
\mu'}))$ and an ideal generated by $\hht \fb$ variables. Hence, by
Part (a), we obtain that $\fb'$ is Cohen-Macaulay and has height
\begin{eqnarray*}
\hht \fb' & = & \hht \fb + \hht \ini (I_2(\bT_{\la' - \mu'})) \\
& = & (n-1) (\la_n - \mu_n -1) + \sum_{i=1}^{n-1} \max \{0, n-1-\mu_i\} + \\
& & \text{ } \sum_{i = 2}^{n-1} \left (\la_i - (\la_n - \mu_n -1) - \max \{n-1, \mu_i\} - 1 \right ) \\
& = & \la_n - \mu_n -1 +  \sum_{i=1}^{n-1} \max \{0, n-1-\mu_i\} +
 \sum_{i = 2}^{n-1} \left (\la_i -   \max \{n-1, \mu_i\} - 1 \right ) \\
 & = & \max \{0, n-1-\mu_1\} + \sum_{i = 2}^n (\la_i - \mu_i -1) \\
 & = & 1 + \hht \ini (I_2(\bSwlm)),
\end{eqnarray*}
where we used the observation that
\[
\max \{0, n-1 - \mu_i\} - \max \{n-1, \mu_i \} = - \mu_i.
\]
We conclude that $\ini (I_2 (\bSlm))$ is a basic double link of
$\fb'$, and   Proposition \ref{prop:bdl} shows that $\ini (I_2(\bSlm)) $ has the claimed height. Here we abuse our notation if $\lambda_n - \mu_n = 1$. Then row $n$ of  the tableaux $\bT_{\tilde{\lambda} - \mu}$ is empty. Thus, the ideal $I_2 (\bS_{\tilde{\lambda} - \mu})$ has the claimed properties by induction on $n$. 

Finally, denote by $\Delta$ the simplicial complex corresponding to $\ini (I_2(\bSlm))$. Equation \eqref{eq:symm-bdl} provides that the simplicial complexes corresponding to $\fb'$ and $\ini (I_2(\bSwlm))$ are the link $\lk_{\Delta} ({n, \la_n})$ and the deletion $\Delta_{- (n, \la_n)}$, respectively. They are both vertex decomposable by the induction hypothesis, and hence so is $\Delta$.
\end{proof}

\begin{Corollary}
  \label{cor:CM}
The rings $K[\bTlm]/I_2 (\bTlm)$ and $K[\bTlm]/I_2 (\bSlm)$ are  Cohen-Macaulay.
\end{Corollary}

\begin{proof}
This follows from the corresponding result for the initial ideals in Theorem \ref{thm:initial-ideals}.
\end{proof}

We use the previous theorem and a well-known localization technique (see, e.g., \cite[Lemma 7.3.3]{BH}) to establish the following: 

\begin{Proposition}
  \label{prop:primality}
The ideals   $I_2 (\bTlm)$ and $I_2 (\bSlm)$ are prime ideals in $K[\bTlm]$.
\end{Proposition}

\begin{proof}
We use again induction based on obtaining $\bTlm$ by adding a new right-most box in the last row of a smaller tableau. So set
\[
\widetilde{\la} := (\la_1,\ldots,\la_{n-1}, \la_n - 1).
\]
By induction, we may assume that $I_2 (\bTwlm)$ and $I_2 (\bSwlm)$ are prime. Since the proof for $I_2 (\bTlm)$ is similar, but easier, we only provide the arguments that $I_2 (\bSlm)$ is a prime ideal.

Consider the $K$-algebra homomorphism $\ffi: K[\bTlm][T_{n, \la_n}^{-1}] \to K[\bTlm][T_{n, \la_n}^{-1}]$ defined by
\[
T_{i j} \mapsto \begin{cases}
  T_{i, j} + T_{n, j} T_{i, \la_n} T_{n, \la_n}^{-1}  & \text{if } (i, j) \neq (n, \la_n) \\
 T_{i, j} & \text{otherwise}.
\end{cases}
\]
In fact, $\ffi$ is an isomorphism whose inverse map is $\psi: K[\bTlm][T_{n, \la_n}^{-1}] \to K[\bTlm][T_{n, \la_n}^{-1}]$ defined by
\[
T_{i j} \mapsto \begin{cases}
  T_{i, j} - T_{n, j} T_{i, \la_n} T_{n, \la_n}^{-1}  & \text{if } (i, j) \neq (n, \la_n) \\
 T_{i, j} & \text{otherwise}.
\end{cases}
\]
Notice that $\ffi$ maps the extension of $I_2 (\bSlm)$ to the extension of $\fb + I_2(\bT_{\la' - \mu'})$, where
\begin{equation*}
\la' := (\la_1 - (\la_n - \mu_n -1),\ldots,\la_{n-1} - (\la_n - \mu_n -1)) \in \ZZ^{n-1}
\end{equation*}
and
\begin{equation*}
\mu' := (\mu_1',\ldots,\mu_{n-1}') \in \ZZ^{n-1}, \quad
\mu_i' := \max \{n-1, \mu_i\},
\end{equation*}
and
$\fb$ is the ideal used in the proof of Theorem \ref{thm:initial-ideals}:
\[
\fb = (T_{i, j} \s 1 \le i < n \text{ and } [ \mu_i < j < n \text{ or } \mu_n < j < \la_n] ).
\]
It follows that we get  isomorphisms
\begin{align*}
  (K[\bTlm]/I_2 (\bSlm))[t_{n, \la_n}^{-1}] & \cong (K[\bTlm]/(\fb + I_2(\bT_{\la' - \mu'})) )[T_{n, \la_n}^{-1}]\\
   & \cong (K[\bT_{\la' - \mu'}, T_{n, \la_n}]/I_2(\bT_{\la' - \mu'}))[T_{n, \la_n}^{-1}],
\end{align*}
where $t_{n, \la_n}^{-1}$ denotes the residue class of $T_{n, \la_n}^{-1}$ in
$A := K[\bTlm]/I_2 (\bSlm)$.

Assume that $A$ is not a domain. Then $I_2 (\bSlm) \neq I_2 (\bSwlm)$, and $A$ has an associated prime ideal that contains  $t_{n, \la_n}$ because the above isomorphisms show that $A[t_{n, \la_n}^{-1}]$ is a domain. Since $A$ is Cohen-Macaulay by Corollary \ref{cor:CM}, all its associated prime ideals have the same height. It follows that the ideals $I_2 (\bSlm)$ and $J := (I_2 (\bSlm), T_{n, \la_n})$ have the same height in $K[\bTlm]$.

Notice that $I_2 (\bSlm) \neq I_2 (\bSwlm)$ implies the existence of a quadratic binomial $f \in I_2 (\bSlm)  \setminus  I_2 (\bSwlm)$ such that $(f, T_{n, \la_n}) = (T_{i, j} T_{k, l}, T_{n, \la_n})$, where $T_{i, j} T_{k, l} \in K[\bTwlm]$. 
By the induction hypothesis, $I_2 (\bSwlm)$ is a prime  ideal generated by quadrics. We conclude that
\[
\hht (I_2 (\bSwlm), T_{i, j} T_{k, l}) = 1 +  \hht I_2 (\bSwlm).
\]
Using Theorem \ref{thm:initial-ideals}, we obtain
\[
\hht (I_2 (\bSwlm), T_{i, j} T_{k, l}) = \hht I_2 (\bSlm).
\]
The ideal on the left-hand side is generated by polynomials in $K[\bTwlm]$. Hence we get \[
\hht J \geq \hht (I_2 (\bSwlm), T_{i, j} T_{k, l}, T_{n, \la_n}) > \hht I_2 (\bSlm).
\]
However, this contradicts the conclusion of the previous paragraph. Hence $A$ is a domain.
\end{proof}

Our results can be partially summarized as follows:

\begin{Corollary}
  \label{cor:primeness}
The rings $K[\bTlm]/I_2 (\bTlm)$ and $K[\bTlm]/I_2 (\bSlm)$ are normal Cohen-Macaulay domains that are Koszul.
\end{Corollary}

\begin{proof}
First, by the two previous results we know that the two rings are Cohen-Macaulay domains.

Second,  since the prime ideals $I_2(\T_{\lambda-\mu})$ and $ I_2(\mathbf S_{\lambda-\mu})$  are generated by binomials they are, in fact, toric ideals (see, e.g., \cite[Proposition 1.1.11]{CLS}). Observe that the initial ideals of $I_2(\T_{\lambda-\mu})$ and $ I_2(\mathbf S_{\lambda-\mu})$ provided by  Theorem~\ref{thm:Minors-Grob}(a) and (b), respectively, are squarefree. It follows that  $K[\bTlm]/I_2 (\bTlm)$ and $K[\bTlm]/I_2 (\bSlm)$ are normal.

Finally, these rings are also Koszul, as $I_2(\T_{\lambda-\mu})$ and $ I_2(\mathbf S_{\lambda-\mu})$ have Gr\"obner bases consisting of quadrics (see \cite[Theorem 2.2]{BHV}). 
\end{proof}

\begin{Remark}
Theorem 3.3 shows in particular that the initial ideal of $I_2(\mathbf S_{\lambda-\mu})$ is glicci. In fact, by a result in \cite{NRo} the ideal $I_2(\mathbf S_{\lambda-\mu})$ itself is glicci. This raises the question whether also ideals generated by minors of higher order than 2 in $\mathbf S_{\lambda-\mu}$ are glicci. Affirmative answers in some cases are established in \cite{NRo}. 
\end{Remark}


\section{Blow-up algebras}\label{sec:blowup}
\label{sec:blow-up-algebras}

We now use the results of the previous sections to elucidate the structure of blow-up algebras of specialized Ferrers ideals. Recall that, for an ideal $I$ in any  commutative ring $R$, its {\em Rees algebra} is the ring $R[I t] = \bigoplus_{j \ge 0} I^j t^j \subset R[t]$, where $t$ is a variable. If $R$ is a graded ring having only one maximal graded ideal, $\fm$, then the {\em special fiber ring} of $I \subset R$ is the algebra
\[
\cF (I) = \bigoplus_{j \ge 0} I^j/\fm I^j \cong R[I t] \otimes_{R} R/\fm.
\]

For a monomial ideal $I$, we denote by $G(I)$ the minimal generating set of $I$ that consists of monomials. If $I$ is a monomial ideal whose minimal generators have degree two, then the special fiber ring $\cF (I)$ is isomorphic to $K[G(I)]$. If the minimal generators of $I$ are even squarefree quadratic monomials, then $I$ is the edge ideal of a simple graph, and $K[G(I)]$ is also called the edge subring of this graph.

In what follows, we determine the special fiber ring of a specialized  Ferrers ideal.
First we find its dimension using results from \cite{vila}. We continue to employ the notation from the previous sections. In particular, $\la$ is a partition with $n$ parts, its largest one being $\la_1 = m$.

\begin{Proposition}
  \label{prop:dim-fibers-ring}
The Krull dimension of the special fiber  ring  of a specialized Ferrers ideal $\overline{I}_{\la - \mu}$ is
\begin{equation*}
  \label{eq:dim-fibers-ring}
\dim \cF (\overline{I}_{\la - \mu}) = m + \min \{0, n-1 - \mu_1 \} =   \begin{cases}
  m & \text{if } \mu_1 \le n-1 \\
  m + n - \mu_1 -1 & \text{if }  \mu_1 \ge n.
\end{cases}
\end{equation*}
\end{Proposition}

\begin{proof} We consider several cases.
Assume that $\mu_1 \ge n$. Then $\overline{I}_{\la - \mu}$ is the edge ideal of a bipartite graph $\Gamma_{\la - \mu}$ on the vertex set $\{x_1,\ldots,x_n\} \sqcup \{x_{\mu_1 + 1},\ldots,x_m\}$. In fact, $\Gamma_{\la - \mu}$ is a Ferrers graph on $n + m-\mu_1$ vertices. Since it is connected, we get $\dim \cF (\overline{I}_{\la - \mu}) = m + n- \mu_1-1$ (see \cite[Proposition 8.2.12]{vila} or \cite{SVV}), as claimed.

Let $\mu_1 \le n-1$. Recall that $\overline{I}_{\la - \mu}$ is not necessarily a squarefree monomial ideal. Consider the subideal of $\overline{I}_{\la - \mu}$ that is generated by the squarefree monomials in $\overline{I}_{\la - \mu}$. It is the edge ideal $\overline{I}_{\la' - \mu'}$ of a connected graph $\Gamma_{\la' - \mu'}$ on $m$ vertices. This is clear if the partition $\la'$ also has $n$ positive parts. However, if the latter condition fails, then $x_n^2$ is in $\overline{I}_{\la - \mu}$. Hence the monomials $x_1 x_n,\ldots, x_{n-1} x_n$ are in $\overline{I}_{\la - \mu}$, so they are in $\overline{I}_{\la' - \mu'}$. It follows that in any case $\Gamma_{\la' - \mu'}$ is a connected graph.

Let $x_j^2$ be a generator of $\overline{I}_{\la - \mu}$ that is not in $\overline{I}_{\la' - \mu'}$. Assume the graph $\Gamma_{\la' - \mu'}$ is not bipartite. Then \cite[Exercise 8.2.16]{vila} implies that $K[G(\overline{I}_{\la' - \mu'}), x_j^2]$ has dimension $m$.
Since
\[
K[G(\overline{I}_{\la' - \mu'}), x_j^2] \subset K[G(\overline{I}_{\la - \mu})] \subset K[x_1,\ldots,x_m], 
\]
we conclude that $K[G(\overline{I}_{\la - \mu})] \cong \cF (\overline{I}_{\la - \mu})$ has dimension $m$.

If $\Gamma_{\la' - \mu'}$ is a bipartite graph, then $K[G(\overline{I}_{\la' - \mu'})]$ has already dimension $m$, and thus the above argument gives again that  the dimension of  $K[G(\overline{I}_{\la - \mu})]$ is $m$, as claimed.
\end{proof}

The main result of this section is:

\begin{Theorem}
   \label{thm:special-fiber-ring-is-ladder-determinantal}
The special fiber ring of  $\overline I_{\lambda-\mu}$ is a  determinantal ring arising from the two-minors of    a symmetric tableau which may have holes. More precisely, there is a graded isomorphism
    \[
       \mathcal F(\overline I_{\lambda-\mu}) \cong  K[\T_{\lambda - \mu}] / I_2 (\bSlm).
    \]
It is a normal Cohen-Macaulay domain that is Koszul. 

\end{Theorem}

\begin{proof}
Consider the algebra epimorphism
\[
\pi\colon K[{\bf T}] \twoheadrightarrow K[G(\overline I_{\lambda-\mu})] \cong
\mathcal F(\overline I_{\lambda-\mu}),
\]
where $\pi(T_{ij})=x_i x_j$. We claim that the kernel of $\pi$ is
the determinantal ideal $I_2 (\bSlm)$. Since $\pi$ maps all 2-minors in $\bSlm$ to zero, we get  $I_2 (\bSlm) \subset \ker \pi$. Both ideals are prime ideals (see Corollary \ref{cor:primeness}). Thus, to deduce the desired equality it is enough to
show that the two ideals have the same height. Then Corollary~\ref{cor:primeness} gives the asserted properties of $\mathcal F(\overline I_{\lambda-\mu})$. 

Theorem \ref{thm:initial-ideals}(b), on the one hand, implies that 
\[
\hht I_2(\bSlm) = \max \{0, n-1-\mu_1\} +  \sum_{i=2}^n (\la_i - \mu_i - 1).
\]

On the other hand,  Proposition \ref{prop:dim-fibers-ring} and $\la_1 = m$ provides 
\begin{align*}
  \hht \ker \pi & = \dim K[\bTlm] - \dim \mathcal F (\overline I_{\lambda-\mu}) \\
  & = \sum_{i=1}^n (\la_i - \mu_i) \; - \; \left [ m + \min \{0, n-1 - \mu_1 \} \right ] \\
  & = \sum_{i=2}^n (\la_i - \mu_i - 1) - \mu_1 + n-1 - \min \{0, n-1 - \mu_1 \} \\
  & =  \sum_{i=2}^n (\la_i - \mu_i - 1) + \max \{0, n-1-\mu_1\},
\end{align*}
as desired.
\end{proof}

\begin{Remark}
  \label{rem:includes-Ferrers-graphs}
Observe that the description of the special fiber ring becomes simpler if $\mu_1 \ge n-1$. Indeed, then  $I_2 (\bSlm) = I_2 (\T_{\lambda-\mu})$ (see Figure \ref{fig:no new minors by symm}). Moreover, if $\mu$ satisfies the even stronger assumption $\mu_1 \ge n$, then $\overline I_{\lambda-\mu}$ is the edge ideal of a Ferrers graph. Thus, Theorem \ref{thm:special-fiber-ring-is-ladder-determinantal} includes in particular a description of the special fibers ring of a Ferrers ideal. This identification was first obtained in \cite[Proposition 5.1(b)]{CorsoNagel}. Note however that the descriptions above and in \cite{CorsoNagel} use a priori different determinantal ideals, due to the presentation of a Ferrers ideal by different tableau.


\begin{figure}[h!]
\includegraphics{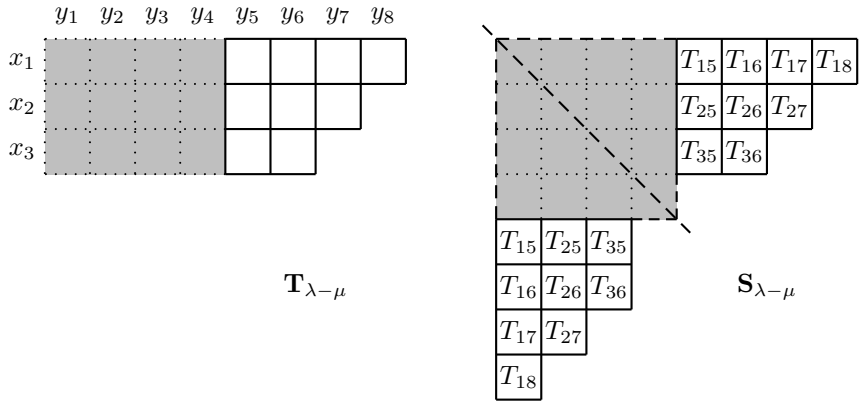}
\caption{A skew shape and its symmetrization for $\la = (8,7,6)$ and $\mu = (4,4,4)$} 
\label{fig:no new minors by symm}
\end{figure}
\end{Remark}

The last result allows us also to give a determinantal description of the Rees algebra of a specialized Ferrers ideal.

\begin{Corollary}
  \label{cor:Rees-algebra}
Let  $\overline I_{\lambda-\mu} \subset R=K[x_1,\ldots,x_m]$ be a specialized Ferrers ideal. Then  its Rees algebra $R[\overline I_{\lambda-\mu} t]$ is isomorphic to
the special fiber ring $\cF (J)$ of the ideal $J \subset R[x_0]$, where $x_0$ is a new variable and $J = \overline I_{\lambda-\mu} + x_0 (x_1,\ldots,x_m) \subset R$.

In particular, the Rees algebra $R[\overline I_{\lambda-\mu} t]$  is a normal Cohen-Macaulay domain that is Koszul.
\end{Corollary}

We prove this result below after making it more precise. 

\begin{Remark}
  \label{rem:Rees-algebra-to-fewer-variables}
Notice that in the case $\mu_1 > n$ none of the variables $x_i$ with $n < i \leq \mu_1$ divides a monomial minimal generator of $\overline I_{\lambda-\mu}$. Thus, the properties of $\overline I_{\lambda-\mu}$ can be studied by considering it as an ideal in the smaller polynomial ring, which is obtained from $R$ by dropping the variables $x_{n + 1},\ldots,x_{\mu_1}$. Equivalently, this amounts to renaming the variables $x_{\mu_1+1},\ldots,x_{m}$ by $x_{n + 1},\ldots,x_{m+n-\mu_1}$ and considering the resulting Ferrers ideal $\overline I_{\lambda'-\mu'}$ in a polynomial ring with variables $x_{1},\ldots,x_{m+n-\mu_1}$, where now $\mu'_1 = n$. This allows us to essentially restrict ourselves to Ferrers ideals $\overline I_{\lambda-\mu}$ satisfying $\mu_1 \leq n$.
\end{Remark}

\begin{Corollary}
  \label{cor:Rees-algebra-simplified}
Let  $\overline I_{\lambda-\mu} \subset R=K[x_1,\ldots,x_m]$ be a specialized Ferrers ideal with $\mu_1 \le n$. Then   its Rees algebra is determinantal. More, precisely, there are algebra isomorphisms
\[
R[\overline I_{\lambda-\mu} t] \cong \cF (\overline I_{\lambda'-\mu'}) \cong K[\bT_{\la' - \mu'}]/I_2 (\bS_{\la' - \mu'}),
\]
where
\[
\la'  = (\la_1 + 1, \la_1 + 1, \la_2 + 1,\ldots,\la_n + 1) \in \ZZ^{n+1}
\]
and
\[
\mu' = (1, \mu_1 + 1, \mu_2 + 1, \mu_n + 1) \in \ZZ^{n+1}.
\]
\end{Corollary}

\begin{Remark}
  \label{rem:Rees-algebra-simplified}

(i) The passage from the special fiber ring of $\overline I_{\lambda-\mu}$ to  its Rees algebra given in  Corollary \ref{cor:Rees-algebra-simplified} can also be described as follows. Augment the tableau $\bSlm$ with a new top row and a new leftmost column. Leave the new northwest corner empty and fill the new top row with the variables $x_1,\ldots,x_m$ from left to right and the leftmost column with $x_1,\ldots,x_m$ from top to bottom. 

\begin{figure}[h!]
\includegraphics{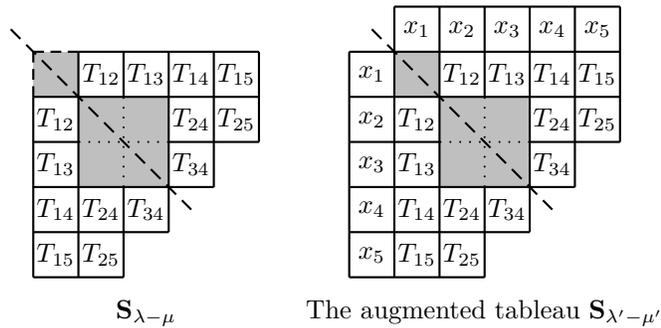}
\caption{A symmetrized tableau and its augmentation.}
\label{fig:SandAugm}
\end{figure}
Let $I$ be the ideal of $R[\bTlm]$ that is generated by the 2-minors in the augmented tableau. Up to the names of the variables, the augmented tableau is the same as $\bS_{\la' - \mu'}$. Hence Corollary \ref{cor:Rees-algebra-simplified} gives the isomorphism
\[
R[\overline I_{\lambda-\mu} t] \cong R[\bTlm]/I.
\]

(ii) If $\overline I_{\lambda-\mu}$ is the edge ideal of a graph $\Gamma$, then the last isomorphism says that the Rees algebra of $\overline I_{\lambda-\mu}$ is isomorphic to the special fiber ring of the edge ideal to the cone over $\Gamma$. This is true for arbitrary edge ideals of graphs by \cite[Proposition 8.2.15]{vila}.

(iii) The Rees algebra of a complete graph on $n$ vertices was already identified by Villarreal (see \cite[Exercise 9.2.14]{vila}). In our notation this is the ring $R[\overline I_{\lambda-\mu} t]$, where $\la = (n,n,\ldots,n) \in \ZZ^n$ and $\mu = (1,2,\ldots,n) \in \ZZ^n$.
\end{Remark}

\begin{proof}[Proof of Corollary \ref{cor:Rees-algebra} and Corollary \ref{cor:Rees-algebra-simplified}]
Consider the ring homomorphisms
\[
\ffi: R[\bTlm] \to R[\overline I_{\lambda-\mu} t], \; T_{i j} \mapsto x_i x_j
\]
and
\[
\alpha:  R[\overline I_{\lambda-\mu} t] \to K[J], \; x_k \mapsto x_0 x_k, \; x_i x_j t \mapsto x_i x_j.
\]
The first part of Corollary \ref{cor:Rees-algebra} follows because $\alpha$ is an isomorphism.

Now let us assume $\mu_1 \le n$. Then, up to renaming variables, the ideals $J$ and $\overline I_{\lambda'-\mu'}$ are equal, and $R[\bT_{\lambda-\mu}]$ is isomorphic to the polynomial ring $K[\bT_{\la' - \mu'}]$. Hence $\alpha \circ \ffi$ is the map that induces an isomorphism $K[\bT_{\la' - \mu'}]/I_2 (\bS_{\la' - \mu'}) \cong K [\overline I_{\lambda'-\mu'}]$, which in turn in isomorphic to $K[J] \cong \cF (J)$. This establishes Corollary \ref{cor:Rees-algebra-simplified}. Moreover, Corollary \ref{cor:primeness} gives that $\cF (J)$ is a normal Cohen-Macaulay domain that is Koszul.

It remains to consider the case $\mu_1 > n$. Put $S = K[x_1,\ldots,x_n, \ x_{\mu_1 + 1},\ldots,x_m]$. Then $\cF (J) \cong \cF (J S)[x_{n+1},\ldots,x_{\mu_1}]$. Since, $\cF (J S)$ is a determinantal ring that is Koszul and a normal Cohen-Macaulay domain by Corollary \ref{cor:Rees-algebra-simplified}, the same is true for $\cF (J)$.
\end{proof}


\section{Minimal reductions}\label{sec:reductions}

In the special case where $\lambda = (m,m,\ldots,m) \in \ZZ^n$, the ideal  $I_{\lambda}$ is the edge ideal of a complete bipartite graph, and a distinguished  minimal reduction of $I_{\lambda}$ is given by the Dedekind-Mertens content formula (see \cite{Northcott, CVV, BG}). 
Here we extend this result to arbitrary Ferrers ideals. 

Recall that an ideal $J$ is said to be a \emph{reduction} of an ideal $I$ if $J \subset I$ and there is an integer $r \ge 0$ such that 
\[
I^{r+1} = J \cdot I^r. 
\]
The minimum integer $r$ such that this equality holds is called the \emph{reduction number} of $I$ with respect to $J$ and denoted by $r_J (I)$. A reduction $J$ is \emph{minimal} if no ideal strictly
contained in $J$ is a reduction of $I$. The (absolute) \emph{reduction number} of $I$ is  
\[
r (I) = \min \{r_J(I) \s J \text{ is a minimal reduction of } I\}.
\]

\begin{Theorem}
  \label{thm:min-red-Ferrers}
For every partition $\lambda = (\lambda_1 = m \ge \lambda_2 \ge \cdots \ge \lambda_n)$ with positive parts, the $m+n-1$ diagonals of the Ferrers tableau $\bT_{\lambda}$  generate a minimal reduction $J_{\lambda}$ of the Ferrers ideal $I_{\lambda}$. More precisely, this minimal reduction is generated by
\[
\sum_{i \ge 1} x_i y_{k+i}, \quad k = 0,\ldots,m-1,
\]
and
\[
\sum_{i \ge 1} x_{k+i} y_{i}, \quad k = 1,\ldots,n-1,
\]
where the summands are monomials that are contained in $I_{\lambda}$.
\end{Theorem}

\newpage
For the partition $\lambda:=(5,5,5,2,1)$, this minimal reduction can be represented by the diagram in Figure~\ref{fig:exMinRed}. 
\begin{figure}[h!]
\includegraphics{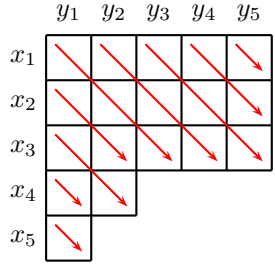}
\caption{Example of Theorem~\ref{thm:min-red-Ferrers}: a Ferrers minimal reduction for $\lambda=(5,5,5,2,1)$.}     \label{fig:exMinRed}
\end{figure}

For the specialized Ferrers ideal $\overline{I}_{\la - \mu}$, we find a distinguished minimal reduction in an important special case, namely when it is a  strongly stable monomial ideal. Figure~\ref{fig:exMinRedSpec} illustrates the result in a simple case. 

\begin{Theorem}
  \label{thm:min-red-specialized-Ferr}
Let $\mu = (0,1,\ldots,n-1) \in \ZZ^n$, and let $\la = (\la_1,\ldots,\la_n)$ be a partition, where $\la_1 = m$ and $\la_n \ge n$. Then the $m$ diagonals in the tableau $\bTlm$ generate a minimal reduction $\overline{J}_{\la - \mu}$ of the specialized Ferrers ideal $\overline{I}_{\la - \mu}$. More precisely, this minimal reduction is generated by
\[
\sum_{i \ge 1} x_i x_{k+i}, \quad k = 0,\ldots,m-1,
\]
where the summands are monomials that are contained in $\overline{I}_{\la - \mu}$.
\end{Theorem}

\begin{figure}[h!] 
\includegraphics{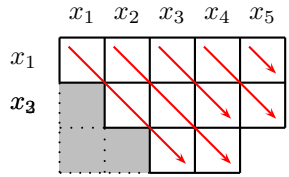}
\caption{An illustration of Theorem~\ref{thm:min-red-specialized-Ferr} for $\lambda=(5,5,4)$.}     \label{fig:exMinRedSpec}
\end{figure}

The proofs of the last two statements are based on results about certain initial ideals.

\begin{Lemma}
   \label{lem:in-Ferrers-red}
Let $I_2 (\bT_{\la}) \subset K[\bT_{\la}]$ be the ideal generated by the 2-minors of $\bT_{\la}$, and let $L \subset K[\bT_{\la}]$ be the ideal generated by the $m+n$ diagonals  
\[ 
\sum_{i \ge 1} T_{i, k+i} \quad ( \text{for } k = 0,\ldots,m-1) \quad \text{ and } \quad  \sum_{i \ge 1} T_{k+i, i}\quad ( \text{for } k = 1,\ldots,n-1). 
\] Then a power of every variable $T_{ij}$ in $\bT_{\la}$ is in an  initial ideal of the
ideal of $I_2 (\bT_{\la})+L$. More precisely, $T_{ij}^j \in \ini_{\prec}(I_2 (\bT_{\la})+L)$,
where $\prec$ is the reverse-lexicographic term order induced by the
row-ordering on the tableau, that is,
\[
    T_{1,1} > T_{1,2}>\dots >T_{1,\lambda_1} >
    T_{2,1} > \dots > T_{2,\lambda_2} >
    \dots > T_{n,\lambda_n}.
\]
\end{Lemma}

\begin{proof}
To simplify notation put $Q = I_2 (\bT_{\la})$.
    In what follows,  diagonals and minors with $T_{ij}$ in their support will be used  to construct a polynomial in the ideal $Q+L$
    whose initial term is $T_{ij}^j$. This condition will be satisfied by ensuring that all other terms are divisible
    either by monomials in the initial ideal or
    by variables that are reverse-lexicographically smaller then $T_{ij}$, i.e.\ are to the east or south of $T_{i j}$. Notice that the initial monomial of each 2-minor is the product of the variables on its antidiagonal. 
    \smallskip

{\em Claim:}  For each variable $T_{i j}$ in $\bT_{\la}$, the following polynomial is in $Q + L$:
    \[ \tag{*}\label{final-polynomial-after-replacements}
       \pm \sum_{p_{1}>0}T_{i-p_1,j} \sum_{p_2>0} T_{i-p_2,j} \cdots \sum_{p_{j-2}>0} T_{i-p_{j-2},j}
                \left(\sum_{p_{j-1}>0}  T_{i-p_{j-1},j} T_{i,j-p_1-\cdots-p_{j-1}}  \right) + T_{ij}^j +  L.O.T.
    \]
Here and below we always use the conventions that ``L.O.T." stands for ``lower-order terms" and represents monomials that are $\prec$-smaller than the last monomial listed (i.e., $T_{i j}^j$ above) and the sums only involve variables that are in $\bT_{\la}$. The latter allows avoiding  specifying the upper limits of the summations explicitly, thus greatly simplifying notation.
\smallskip

Let $D_{ij}\in L$ be the diagonal passing through $T_{ij}$, that is,
\[
D_{i j} = \sum_{p > 0} T_{i-p, j-p} + T_{i j} + L.O.T.
\]
Furthermore, let $Q_{i,j;i-p,j-p}\in Q$ be the 2-minor of $\bT_{\la}$ whose diagonal term is $T_{ij}T_{i-p,j-p}$.

Using $T_{1 j}^{j-1} D_{1 j}$, the claim is true if $i = 1$. Let $i > 1$. Then $T_{ij}^{j-1} D_{ij}$ reads as
\[\tag{d1}\label{use-first-diagonal}
        T_{ij}^{j-1} D_{ij} = T_{ij}^{j-1} \sum_{p_1 > 0} T_{i-p_1, j-p_1} + T_{ij}^j + L.O.T.
    \]
Thus, we are done if $j = 1$.  Let $ j > 1$. Continue to  successively modify the above polynomial by replacing variables $T_{k l}$ that are above and strictly to the left of $T_{ij}$  by using the diagonal $D_{k l}$ if $k = i$  and by using the minor $Q_{i, j; k ,l}$ if $k < i$.

Following this strategy,  subtract suitable multiples of the minors $Q_{i,j;i-p_1,j-p_1}$ from the polynomial \eqref{use-first-diagonal} and obtain
\[
      T_{ij}^{j-2}   \sum_{p_1>0} T_{i-p_1,j} T_{i,j-p_1}  + T_{ij}^j  + L.O.T. \in Q + L.
\]

If $j = 2$, this shows the claim. Otherwise,  repeat the process. In order to substitute the variables $T_{i,j-p_1}$, use the diagonals $D_{i,j-p_1} = \sum_{p_2>0} T_{i-p_2,j-p_2-p_1} + T_{i,j-p_1} + L.O.T.$ Subtracting suitable multiples of them provides 
\[
- T_{ij}^{j-2} \sum_{p_1>0} T_{i-p_1,j} \sum_{p_2>0} T_{i-p_2,j-p_2-p_1} + T_{ij}^j + L.O.T. \in Q + L.
\]
Next,  subtract suitable multiples of the minors $Q_{i,j;i-p_2,j-p_2-p_1}$ ( $p_1,p_2>0$), obtaining 
    \[\tag{q1}\label{use-second-minor}
        - T_{ij}^{j-3} \sum_{p_1>0} T_{i-p_1,j}\sum_{p_2>0} T_{i-p_2,j} T_{i,j-p_1-p_2} + T_{ij}^j +  L.O.T. \in Q + L.
    \]
This gives the claim if $ j = 3$.
    In general, repeating the process $j-1$ times provides the following polynomial in $L+Q$:
    \[
       \pm \sum_{p_1>0}T_{i-p_1,j} \sum_{p_2>0} T_{i-p_2,j} \cdots \sum_{p_{j-2}>0} T_{i-p_{j-2},j}
                 \sum_{p_{j-1}>0}  T_{i-p_{j-1},j} T_{i,j-p_1-\cdots-p_{j-1}}  + T_{ij}^j +  L.O.T.
    \]
This establishes the claim in general. 
\smallskip

Finally, observe that $p_1,\dots,p_{j-1}>0$ implies $p_1+\cdots p_{j-1} \ge j-1$. Hence the polynomial \eqref{final-polynomial-after-replacements} can be rewritten as
\[
\pm T_{i-1, j}^{j-1} T_{i 1} + T_{ij}^j +  L.O.T.
\]
Subtracting the appropriate multiple of the diagonal $D_{i,1}$ results in a polynomial whose leading terms is $T_{i j}^j$. This completes our argument.
\end{proof}

For a strongly stable specialized Ferrers ideal, an analogous result holds. 

\begin{Lemma}
   \label{lem:in-genFerrers-red}
Let $I_2 (\bSlm) \subset K[\bT_{\la - \mu}]$ be the ideal generated by the 2-minors of $\bSlm$, where $\mu = (0,1,\ldots,n-1) \in \ZZ^n$,  and let $L \subset K[\bTlm]$ be the ideal generated by the $m$ diagonals  $\sum_{i \ge 1} T_{i, k+i}$ $(k = 0,\ldots,m-1)$. Then a power of every variable $T_{ij}$ in $\bTlm$ is in an initial ideal of the
ideal of $I_2 (\bSlm)+L$. More precisely, $T_{ij}^j \in \ini_{\prec}(I_2 (\bSlm)+L)$,
where $\prec$ is the reverse-lexicographic term order induced by the
row-ordering on the tableau, that is,
\[
    T_{1,\mu_1+1} > T_{1,\mu_1+2}>\dots >T_{1,\lambda_1} >
    T_{2,\mu_2+1} > \dots > T_{2,\lambda_2} >
    \dots > T_{n,\lambda_n}.
\]
\end{Lemma}

\begin{proof}
The proof is completely analogous to the argument used to establish Lemma \ref{lem:in-Ferrers-red}.
\end{proof}

The main results of this section follow now easily.

\begin{proof}[Proof of Theorem \ref{thm:min-red-Ferrers}]
Lemma \ref{lem:in-Ferrers-red} shows that the radical of the ideal $I_2 (\bT_{\la}) + L$ is generated by the variables in $\bT_{\la}$. Since the special fiber  ring $\mathcal F( I_{\lambda})$  of $I_{\la}$ has dimension $m+n$ by \cite{SVV} and is isomorphic to $K[\bT_{\la}]/I_2 (\bT_{\la})$ by \cite[Proposition 5.1]{CorsoNagel}, it follows that the diagonals generating  $L$ form a system of parameters of $\mathcal F( I_{\lambda})$. Hence, the claim follows (see, e.g., \cite[Proposition 8.2.4]{HS}).
\end{proof}

Analogous arguments, using Lemma \ref{lem:in-genFerrers-red} and Theorem \ref{thm:special-fiber-ring-is-ladder-determinantal}, provide the {\em Proof of Theorem \ref{thm:min-red-specialized-Ferr}}. In the interest of space, the details are omitted.

\begin{Remark}\label{rmk:OtherMinRed}
(i) Smith  used Theorem \ref{thm:min-red-specialized-Ferr} to compute the core of certain Ferrers ideals, that is, the intersection over all minimal reductions of such a Ferrers ideal (see \cite[Theorem 5.1]{Smith}).

(ii)
It would be desirable to extend Theorem \ref{thm:min-red-specialized-Ferr}, that is, to find a distinguished minimal reduction of other specialized Ferrers ideals. Notice that the diagonals in the tableau $\bTlm$ do not generate a minimal reduction of $\overline{I}_{\la - \mu}$ in general. In fact, if $\mu_1 \ge 1$, then the number of diagonals is less than the number of generators of any minimal reduction of $\overline{I}_{\la - \mu}$.
\end{Remark}

\begin{Example} \label{ex:min-red}
Consider the specialized Ferrers ideals associated to $\la = (4,4,4)$ and $\mu = (1, 2, 3)$.
It is
\[
\overline{I}_{\la - \mu} = (x_1 x_2, x_1 x_3, x_1 x_4, x_2 x_3, x_2, x_2 x_4, x_3 x_4) \subset K[x_1, x_2, x_3, x_4].
\]
According to Proposition \ref{eq:dim-fibers-ring}, its special fiber ring has dimension four. Thus, every minimal reduction of $\overline{I}_{\la - \mu}$ has four minimal generators.

\begin{figure}[h!]
\includegraphics{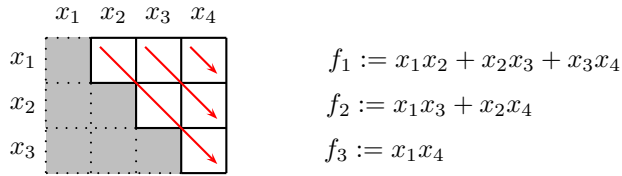}
\caption{Illustration of Example~\ref{ex:min-red}.}
\end{figure}
Since the tableau $\bTlm$ has only three diagonals,  another generator is needed!  Indeed, one can check that the three diagonals together with the polynomial 
\[
x_1 x_2 +  x_1 x_3 + x_1 x_4 +  x_2 x_4 + x_3 x_4
\]
generate a minimal reduction of $\overline{I}_{\la - \mu}$.
\end{Example}

This and other examples suggest that each specialized Ferrers ideal $\overline{I}_{\la - \mu}$ has a minimal reduction consisting of the diagonals in $\bTlm$ and suitably many additional generators. However, we have not been able to find combinatorial descriptions for the needed additional generators.


\section{Hilbert functions and reduction numbers} 
\label{sec:numbers}

We now determine the Hilbert function of the determinantal rings introduced in Section \ref{sec:symmetric-ladders-G-bases}. This allows us to find their Castelnuovo-Mumford regularity. We then show that this regularity gives the reduction number of the Dedekind-Mertens-like reductions we established in the previous section. We conclude with some examples to illustrate our results. 
 
 In order to compute the Hilbert series of the special fiber rings of the specialized Ferrers ideals, we first establish a recursive formula using the Gorenstein liaison results proven in Section \ref{sec:symmetric-ladders-invariants}. This is similar to the approach used in \cite{CorsoNagel}.

Recall from the previous section that
\[
{\mathcal
F}(\overline{I}_{\lambda - \mu}) \cong K[\overline{I}_{\lambda - \mu}] \cong K[\bTlm]/I_2 (\bSlm)
\]
has dimension $m + \max \{0, n-1-\mu_1\}$.  Hence there is a  unique
polynomial $\overline{p}_{\lambda - \mu} \in \ZZ[t]$ such that the
Hilbert series of $K[\overline{I}_{\lambda - \mu}]$ can be written
as
\[
H_{K[\overline{I}_{\lambda - \mu}]} (t) = \frac{\overline{p}_{\lambda - \mu} (t)}{(1-t)^{m + \max \{0, n-1-\mu_1\}}}
\]
and $e (K[\overline{I}_{\lambda - \mu}]) = \overline{p}_{\lambda -
\mu} (1) > 0$ is the multiplicity of $K[\overline{I}_{\lambda -
\mu}]$. The polynomial $\overline{p}_{\lambda - \mu}$ is called the normalized numerator of the Hilbert series. 
Using this notation allows us to state  the desired recursion
formula.  We continue to use the notation and assumptions introduced
at the beginning of Section  \ref{sec:symmetric-ladders-G-bases}.

\begin{Lemma}
   \label{lem:liai-recursion}
Given $\la = (\la_1,\ldots,\la_{n-1}, \la_n) \in \ZZ^n$ and $\mu = (\mu_1,\ldots,\mu_n) \in \ZZ^n$,  set
\[
\widetilde{\la} := (\la_1,\ldots,\la_{n-1}, \la_n - 1) \in \ZZ^n,
\]
\[
\la' := (\la_1 - (\la_n - \mu_n -1),\ldots,\la_{n-1} - (\la_n - \mu_n -1)) \in \ZZ^{n-1},
\]
and
\[
\mu' := (\mu_1',\ldots,\mu_{n-1}') \in \ZZ^{n-1}, \; \text{ where }
\mu_i' := \max \{n-1, \mu_i\}.
\]
If $n \geq 2$, then there is the
following relation among Hilbert series$:$

\begin{equation*}
\overline{p}_{\lambda - \mu} (t)  = \begin{cases} 
\overline{p}_{\widetilde{\lambda} - \mu}(t)  & \text{if } \lambda_2 = \mu_2 + 1 \text{ and } \mu_1 \ge n-1 \\
 \overline{p}_{\widetilde{\lambda} - \mu}(t)  + t \cdot \overline{p}_{\lambda' - \mu'} (t) & \text{otherwise}.
 \end{cases}
\end{equation*}
\end{Lemma}

\begin{proof}
In the proof of Theorem \ref{thm:initial-ideals} (see Equation \eqref{eq:symm-bdl}) we have shown that
\[
\ini (I_2(\bSlm)) = T_{n, \la_n} \fb' + \ini (I_2(\bSwlm)).
\]
Observe that the height of  $\fb'$ is zero if and only if $\lambda_2 = \mu_2 + 1$ and $\mu_1 \ge n-1$. (This follows from the computation at the end of the proof of Theorem \ref{thm:initial-ideals}). This implies the claim in this case since Hilbert functions do not change when passing to the  initial
ideal. 

In the other case, where $\lambda_2 \ge  \mu_2 + 2$ or $\mu_1 \le n-2$, 
apply Proposition \ref{prop:bdl} to conclude that
\[
H_{\cF(\overline{I}_{\lambda - \mu})} (t) = (1-t) \cdot H_{K[\bTlm]/I_2 (\bSwlm)} (t)
      + t \cdot H_{K[\bTlm]/\fb'} (t).
\]
In the proof of Theorem \ref{thm:initial-ideals} we also showed that
$K[\bTlm]/I_2 (\bSwlm) \cong \cF(\overline{I}_{\widetilde{\lambda} -
\mu}) [T_{n, \la_n}]$ and that $K[\bTlm]/\fb'$ has the same Hilbert
series as a polynomial ring over $\cF(\overline{I}_{\lambda' -
\mu'}) = \cF(I_{\lambda' -
\mu'})$. Setting for simplicity $d := \dim K[\overline{I}_{\lambda -
\mu}]$, it follows that
\[
\frac{\overline{p}_{\lambda - \mu}(t)}{(1-t)^d} =
  (1 - t) \cdot \frac{\overline{p}_{\widetilde{\lambda} - \mu}(t)}{(1-t)^{d+1}} +
  t \cdot \frac{\overline{p}_{\lambda' - \mu'}(t)}{(1-t)^d},
\]
which proves our claim.
\end{proof}

In order to compare the results for the special fiber rings of
Ferrers  ideals and their specializations, recall that $\cF(I_{\lambda
- \mu})$ has dimension $m+n-1$. Hence its Hilbert series can be written
as
\[
H_{\cF(I_{\lambda - \mu})} (t) = \frac{p_{\lambda - \mu} (t)}{(1-t)^{m+n-1}}
\]
and $e (\cF(I_{\lambda - \mu})) = p_{\lambda - \mu} (1) > 0$ is the
multiplicity of $K[I_{\lambda - \mu}]$. Again, refer to $p_{\lambda - \mu}$ as the normalized numerator of the Hilbert series. 

We are now ready to derive an explicit formula for the Hilbert series.
Observe that all terms are non-negative. Notice also that in case $n
=1$, the special fiber rings are just polynomial rings over the field  $K$.

\begin{Theorem}
  \label{thm:Hilb-series}
Assume $n \ge 2$. Then:
\begin{itemize}
  \item[(a)] The normalized numerator of the  Hilbert series of
$\cF(I_{\lambda - \mu})$ is:
\[
p_{\lambda} (t) = 1 + h_1 (\lambda - \mu) \cdot t + \cdots + h_{n-1}
(\lambda - \mu) \cdot  t^{n-1},
\]
where
\begin{equation*} 
h_1 (\lambda - \mu ) = \sum_{j=2}^n (\lambda_j -\mu_j - 1)
\end{equation*}
and
\begin{equation*} 
h_k (\lambda - \mu) = \sum_{2 \leq i_1 < i_2 < \ldots < i_k \leq n} \
\sum_{\split{j_{k-1} =  \lambda_{i_1} - \mu_{i_1}}{ - \lambda_{i_k} + \mu_{i_k}  - k +
2}}^{\lambda_{i_1} - \mu_{i_1} - k} \
\sum_{\split{j_{k-2} =  \lambda_{i_1} - \mu_{i_1}}{ - \lambda_{i_{k-1}} + \mu_{i_{k-1}} - k + 3}}^{j_{k-1}} \ldots
\sum_{\split{j_1 = \lambda_{i_1} - \mu_{i_1}}{ - \lambda_{i_{2}} + \mu_{i_2}}}^{j_{2}} j_1,
\end{equation*}
if $k \geq 2$.

  \item[(b)] The normalized numerator of the  Hilbert series of
$\cF(\overline{I}_{\lambda - \mu})$ is:
\[
\overline{p}_{\lambda - \mu} (t) = 1 + \overline{h_1} (\lambda - \mu) \cdot t + \cdots + \overline{h}_{n-1} (\lambda - \mu) \cdot  t^{n-1},
\]
where
\begin{equation*} 
\overline{h}_1 (\lambda - \mu ) = \max \{0, n-1-\mu_1 \} + \sum_{j=2}^n (\lambda_j -\mu_j - 1), 
\end{equation*}
\begin{equation*}
\sigma_j = \begin{cases}
1 & \text{if } j > 0\\
0 & \text{if } j \le 0. 
\end{cases}, 
\end{equation*}
and
\begin{eqnarray*} 
\lefteqn{ \overline{h}_k (\lambda - \mu) }\\
& = & \sum_{2 \leq i_1 < i_2 < \ldots < i_k \leq n} \
\sum_{\split{j_{k-1} =  \lambda_{i_1} - \max\{i_k - 1, \mu_{i_1}\}}{ - \lambda_{i_k} + \max\{i_k - 1, \mu_{i_k}\}  - k + 2}}^{\split{\lambda_{i_1} - \max\{i_k - 1, \mu_{i_1}\}}{ - k + \sigma_{i_k - 1 - \mu_1}}} \
\sum_{\split{j_{k-2} =  \lambda_{i_1} - \max\{i_k - 1, \mu_{i_1}\}}{ - \lambda_{i_{k-1}} + \max\{i_k - 1, \mu_{i_{k-1}} \} - k + 3}}^{j_{k-1}} \ldots
\sum_{\split{j_1 = \lambda_{i_1} - \max\{i_k - 1, \mu_{i_1}\}}{ - \lambda_{i_{2}} + \max\{i_k - 1, \mu_{i_2}\}}}^{j_{2}} j_1,
\end{eqnarray*}
if $k \geq 2$.
\end{itemize}
\end{Theorem}

\begin{proof}
The proof of Claim (a) is similar and only easier than the one of
Claim (b) (see also Theorem 5.4 in \cite{CorsoNagel}). We
restrict ourselves to showing (b) for the case  $\mu_1 \le n-2$. If $\mu_1 \ge n-1$, then each $ \max\{i_k - 1, \mu_{i_1}\}$ in the asserted formula equals $\mu_{i_1}$, so the formula becomes the same as the one in (a). This is correct as $I_2 (\bTlm) = I_2 (\bSlm)$ if $\mu_1 \ge n-1$ (see Remark~\ref{rem:includes-Ferrers-graphs}). 

Assume $\mu_1 \le n-2$. Continue to use the notation introduced in Lemma \ref{lem:liai-recursion}. This result implies for all integers $k \ge 0$
\begin{equation}
  \label{eq:recursion-coefficients}
\overline{h}_k (\lambda - \mu) = \overline{h}_k (\widetilde{\lambda} - \mu) + \overline{h}_{k-1} (\lambda' - \mu').
\end{equation}

A straightforward computation shows that this recursion provides
the claimed formula for $\overline{h}_1 (\lambda - \mu)$. Thus, it
suffices to consider $k \ge 2$.

Now we use induction on $n \ge 2$. If $n = 2$, then
$\overline{p}_{\lambda' - \mu'} = 1$. Thus $\overline{h}_2 (\lambda - \mu) = 0$ by
induction on $\lambda_2 - \mu_2 \ge 1$, using Lemma
\ref{lem:liai-recursion}.

Let $n \geq 3$. Now, we use induction  on $k \geq 2$. Since the case
$k = 2$ is similar, but easier than the general case, we present the
argument only if $k \geq 3$. Finally, we use induction on $\lambda_n
- \mu_n \geq 1$. 

Assume $\la_n - \mu_n = 1$. Then row $n$ in the tableaux $\bT_{\tilde{\lambda} - \mu}$ is empty, so we know $\overline{h}_k (\widetilde{\lambda} - \mu)$ and $\overline{h}_{k-1} (\lambda' - \mu')$ by induction on $n$. Hence, Equation~\eqref{eq:recursion-coefficients} gives 
\begin{eqnarray*}
\lefteqn{ \overline{h}_k (\lambda - \mu) } \\
& = & \sum_{2 \leq i_1 <  \ldots < i_{k} \le n-1} \
\sum_{\split{j_{k-1} =  \lambda_{i_1} - \max\{i_k - 1, \mu_{i_1}\}}{ - \lambda_{i_k} + \max\{i_k - 1, \mu_{i_k}\}  - k + 2}}^{\split{\lambda_{i_1} - \max\{i_k - 1, \mu_{i_1}\}}{ - k + \sigma_{i_k - 1 - \mu_1}}} \
\sum_{\split{j_{k-2} =  \lambda_{i_1} - \max\{i_k - 1, \mu_{i_1}\}}{ - \lambda_{i_{k-1}} + \max\{i_k - 1, \mu_{i_{k-1}} \} - k + 3}}^{j_{k-1}} \ldots
\sum_{\split{j_1 = \lambda_{i_1} - \max\{i_k - 1, \mu_{i_1}\}}{ - \lambda_{i_{2}} + \max\{i_k - 1, \mu_{i_2}\}}}^{j_{2}} j_1 \\
& & + \sum_{2 \leq i_1 <  \ldots < i_{k-1} \le n-1} \
\sum_{\split{j_{k-2} =  \lambda_{i_1} - \max\{n - 1, \mu_{i_1}\}}{ - \lambda_{i_{k-1}} + \max\{n - 1, \mu_{i_{k-1}}\}  - k + 3}}^{\split{\lambda_{i_{1}} - \max\{n - 1, \mu_{i_{1}}\}}{ - k + 1}} \
\sum_{\split{j_{k-3} =  \lambda_{i_1} - \max\{i_k - 1, \mu_{i_1}\}}{ - \lambda_{i_{k-2}} + \max\{i_k - 1, \mu_{i_{k-2}} \} - k + 4}}^{j_{k-2}} \ldots
\sum_{\split{j_1 = \lambda_{i_1} - \max\{i_k - 1, \mu_{i_1}\}}{ - \lambda_{i_{2}} + \max\{i_k - 1, \mu_{i_2}\}}}^{j_{2}} j_1.
\end{eqnarray*}
Observing that $\sigma_{n-1-\mu_1} = 1$ and $\lambda_n = \max\{n-1, \mu_n\} = \lambda_n - \mu_n = 1$, the second summation can be re-written as 
\[
\sum_{2 \leq i_1 <  \ldots < i_{k-1} < i_k = n} \
\sum_{\split{j_{k-1} =  \lambda_{i_1} - \max\{n - 1, \mu_{i_1}\}}{ - \lambda_{n} + \max\{n - 1, \mu_{n}\} - k + 2}}^{\split{\lambda_{i_1} - \max\{n - 1, \mu_{i_1}\}}{ - k + \sigma_{n - 1 - \mu_1}}} \
\sum_{\split{j_{k-2} =  \lambda_{i_1} - \max\{n - 1, \mu_{i_1}\}}{ - \lambda_{i_{k-1}} + \max\{n - 1, \mu_{i_{k-1}}\}  - k + 3}}^{j_{k-1}} \
\ldots
\sum_{\split{j_1 = \lambda_{i_1} - \max\{n - 1, \mu_{i_1}\}}{ - \lambda_{i_{2}} + \max\{n - 1, \mu_{i_2}\}}}^{j_{2}} j_1.
\]
Substituting this into the previous equation gives 
\begin{eqnarray*}
\lefteqn{ \overline{h}_k (\lambda - \mu) } \\
& = & \sum_{2 \leq i_1 <  \ldots < i_{k} \le n} \
\sum_{\split{j_{k-1} =  \lambda_{i_1} - \max\{i_k - 1, \mu_{i_1}\}}{ - \lambda_{i_k} + \max\{i_k - 1, \mu_{i_k}\}  - k + 2}}^{\split{\lambda_{i_1} - \max\{i_k - 1, \mu_{i_1}\}}{ - k + \sigma_{i_k - 1 - \mu_1}}} \
\sum_{\split{j_{k-2} =  \lambda_{i_1} - \max\{i_k - 1, \mu_{i_1}\}}{ - \lambda_{i_{k-1}} + \max\{i_k - 1, \mu_{i_{k-1}} \} - k + 3}}^{j_{k-1}} \ldots
\sum_{\split{j_1 = \lambda_{i_1} - \max\{i_k - 1, \mu_{i_1}\}}{ - \lambda_{i_{2}} + \max\{i_k - 1, \mu_{i_2}\}}}^{j_{2}} j_1, 
\end{eqnarray*}
as claimed. 

Assume now $\la_n - \mu_n \ge 2$. Then the induction
hypotheses and Formula \eqref{eq:recursion-coefficients} provide  the following, after 
considering separately the cases $i_k < n$ and $i_k = n$ in the formula for $\overline{h}_k (\widetilde{\lambda} - \mu)$:
\begin{eqnarray*}
  \label{eq:sums}
\lefteqn{ \overline{h}_k (\lambda - \mu) } \\
& = &  \sum_{2 \leq i_1 <  \ldots < i_{k} \le n-1} \
\sum_{\split{j_{k-1} =  \lambda_{i_1} - \max\{i_k - 1, \mu_{i_1}\}}{ - \lambda_{i_k} + \max\{i_k - 1, \mu_{i_k}\}  - k + 2}}^{\split{\lambda_{i_1} - \max\{i_k - 1, \mu_{i_1}\}}{ - k + \sigma_{i_k - 1 - \mu_1}}} \
\sum_{\split{j_{k-2} =  \lambda_{i_1} - \max\{i_k - 1, \mu_{i_1}\}}{ - \lambda_{i_{k-1}} + \max\{i_k - 1, \mu_{i_{k-1}} \} - k + 3}}^{j_{k-1}} \ldots
\sum_{\split{j_1 = \lambda_{i_1} - \max\{i_k - 1, \mu_{i_1}\}}{ - \lambda_{i_{2}} + \max\{i_k - 1, \mu_{i_2}\}}}^{j_{2}} j_1 \\
& & + \sum_{2 \leq i_1 <  \ldots < i_{k-1} <  i_k = n} \
\sum_{\split{j_{k-1} =  \lambda_{i_1} - \max\{n - 1, \mu_{i_1}\}}{ - (\lambda_{n} - 1) + \max\{n - 1, \mu_{n}\}  - k + 2}}^{\lambda_{i_1} - \max\{n - 1, \mu_{i_1}\} - k + 1} \
\sum_{\split{j_{k-2} =  \lambda_{i_1} - \max\{n - 1, \mu_{i_1}\}}{ - \lambda_{i_{k-1}} + \max\{n - 1, \mu_{i_{k-1}} \} - k + 3}}^{j_{k-1}} \ldots
\sum_{\split{j_1 = \lambda_{i_1} - \max\{n - 1, \mu_{i_1}\}}{ - \lambda_{i_{2}} + \max\{n - 1, \mu_{i_2}\}}}^{j_{2}} j_1 \\
& & + \sum_{2 \leq i_1 <  \ldots < i_{k-1} \le n-1} \
\sum_{\split{j_{k-2} =  \lambda_{i_1} - \max\{n - 1, \mu_{i_1}\}}{ - \lambda_{i_1} + \max\{n - 1, \mu_{i_1}\}  - k + 3}}^{\split{\lambda_{i_1} - (\la_n - \mu_n -1) }{ \max\{n - 1, \mu_{i_1}\} - (k-1)}} \
\sum_{\split{j_{k-3} =  \lambda_{i_1} - \max\{n - 1, \mu_{i_1}\}}{ - \lambda_{i_{k-2}} + \max\{n - 1, \mu_{i_{k-2}} \} - k + 4}}^{j_{k-2}} \ldots
\sum_{\split{j_1 = \lambda_{i_1} - \max\{n - 1, \mu_{i_1}\}}{ - \lambda_{i_{2}} + \max\{n - 1, \mu_{i_2}\}}}^{j_{2}} j_1. 
\end{eqnarray*}
Notice that in the third summation the upper limit for $j_{k-2}$ is one less than the lower limit for $j_{k-1}$ in the second summation. Thus,  combining these two summations provides: 
\begin{eqnarray*}
\lefteqn{ \overline{h}_k (\lambda - \mu) } \\
& = & \sum_{2 \leq i_1 <  \ldots < i_{k} \le n-1} \
\sum_{\split{j_{k-1} =  \lambda_{i_1} - \max\{i_k - 1, \mu_{i_1}\}}{ - \lambda_{i_k} + \max\{i_k - 1, \mu_{i_k}\}  - k + 2}}^{\split{\lambda_{i_1} - \max\{i_k - 1, \mu_{i_1}\}}{ - k + \sigma_{i_k - 1 - \mu_1}}} \
\sum_{\split{j_{k-2} =  \lambda_{i_1} - \max\{i_k - 1, \mu_{i_1}\}}{ - \lambda_{i_{k-1}} + \max\{i_k - 1, \mu_{i_{k-1}} \} - k + 3}}^{j_{k-1}} \ldots
\sum_{\split{j_1 = \lambda_{i_1} - \max\{i_k - 1, \mu_{i_1}\}}{ - \lambda_{i_{2}} + \max\{i_k - 1, \mu_{i_2}\}}}^{j_{2}} j_1 \\
& & + \sum_{2 \leq i_1 <  \ldots < i_{k-1} < i_k = n} \
\sum_{\split{j_{k-1} =  \lambda_{i_1} - \max\{n - 1, \mu_{i_1}\}}{ - \lambda_{n} + \max\{n - 1, \mu_{n}\}  - k + 2}}^{\lambda_{i_1} - \max\{n - 1, \mu_{i_1}\} - k + 1} \
\sum_{\split{j_{k-2} =  \lambda_{i_1} - \max\{n - 1, \mu_{i_1}\}}{ - \lambda_{i_{k-1}} + \max\{n - 1, \mu_{i_{k-1}} \} - k + 3}}^{j_{k-1}} \ldots
\sum_{\split{j_1 = \lambda_{i_1} - \max\{n - 1, \mu_{i_1}\}}{ - \lambda_{i_{2}} + \max\{n - 1, \mu_{i_2}\}}}^{j_{2}} j_1 \\
& = & \sum_{2 \leq i_1 < i_2 < \ldots < i_k \leq n} \
\sum_{\split{j_{k-1} =  \lambda_{i_1} - \max\{i_k - 1, \mu_{i_1}\}}{ - \lambda_{i_k} + \max\{i_k - 1, \mu_{i_k}\}  - k + 2}}^{\split{\lambda_{i_1} - \max\{i_k - 1, \mu_{i_1}\}}{ - k + \sigma_{i_k - 1 - \mu_1}}}\
\sum_{\split{j_{k-2} =  \lambda_{i_1} - \max\{i_k - 1, \mu_{i_1}\}}{ - \lambda_{i_{k-1}} + \max\{i_k - 1, \mu_{i_{k-1}} \} - k + 3}}^{j_{k-1}} \ldots
\sum_{\split{j_1 = \lambda_{i_1} - \max\{i_k - 1, \mu_{i_1}\}}{ - \lambda_{i_{2}} + \max\{i_k - 1, \mu_{i_2}\}}}^{j_{2}} j_1,
\end{eqnarray*}
where we used the assumption $\mu_1 \le n-2$ to conclude that $\sigma_{i_k - 1 - \mu_1} = 1$ if $i_k = n$. 
This completes the proof.
\end{proof}

\begin{Corollary}
  \label{cor:non-vanishing}
Assume $n \ge 2$. Then, for any integer $k \in \{1,\ldots,n-1\}$:
\begin{itemize}
  \item[(a)] ${h}_k (\lambda - \mu)
  >0$ if and only if
  \[
k \le \la_i - \mu_i + i - 3 \quad \text{for all } i = 2,\ldots,k+1.
  \]

  \item[(b)] $\overline{h}_k (\lambda - \mu)
  >0$ if and only if there is some integer $i_k \in \{k+1,\ldots,n\}$ such that 
\begin{equation}
     \label{eq:top condition}
     2 \le \la_{i_k} - \mu_{i_k}  + \sigma_{i_k - 1 - \mu_1}
\end{equation}  
and   
\begin{equation}
     \label{eq:other cond}
k \le \la_i - \max \{i_k - 1, \mu_i\} + i - 3 +  \sigma_{i_k - 1 - \mu_1} \quad \text{for all }
i = 2,\ldots,k+1.
\end{equation}
\end{itemize}
\end{Corollary}

\begin{proof}
 First, let us show (b). If $2 \le k < n$, then the formula for $\overline{h}_k (\lambda - \mu)$ gives that $\overline{h}_k (\lambda - \mu)$ is positive if and only if there are integers $i_2 < i_3 < \cdots < i_k$ in $\{2,\ldots,n\}$ such that the number $\lambda_{i_1} - \max\{i_k - 1, \mu_{i_1}\} - k + \sigma_{i_k - 1 - \mu_1}$ is positive and not less than each $\lambda_{i_1} - \max\{i_k - 1, \mu_{i_1}\} - \lambda_{i_j} + \max\{i_k - 1, \mu_{i_j}\}  - j + 2$ whenever $2 \le j \le k$. This is equivalent to 
 \begin{equation}
     \label{eq:equiv cond}
 k \le \lambda_{i_j} - \max\{i_k - 1, \mu_{i_j}\} + j - 2 + \sigma_{i_k - 1 - \mu_1} \text{ for all } j = 1,\ldots,k. 
 \end{equation}
For $j = k$ this condition becomes Inequality \eqref{eq:top condition} because $\mu_{i_k} \ge i_k -1$. Furthermore, for each $j \in \{1,\ldots,k\}$, we have $i_j \ge j+1$. Thus, using \eqref{eq:equiv cond} we obtain
\begin{align*}
k & \le \lambda_{i_j} - \max\{i_k - 1, \mu_{i_j}\} + j - 2 + \sigma_{i_k - 1 - \mu_1}  \\
& \le \lambda_{j+1} - \max\{i_k - 1, \mu_{j+1}\} + (j+1) - 3 + \sigma_{i_k - 1 - \mu_1}.  
\end{align*}
Hence, we have shown that Conditions \eqref{eq:equiv cond} imply \eqref{eq:top condition} and \eqref{eq:other cond}. 

Conversely, assume \eqref{eq:top condition} and \eqref{eq:other cond} are satisfied. Choosing then 
$i_j = j+1$ for $j = 1,\ldots,k-1$, we obtain
\begin{align*}
k & \le \lambda_{j+1} - \max\{i_k - 1, \mu_{j+1}\} + (j+1) - 3 + \sigma_{i_k - 1 - \mu_1} \\
& = \lambda_{i_j} - \max\{i_k - 1, \mu_{i_j}\} + j - 2 + \sigma_{i_k - 1 - \mu_1}. 
\end{align*} 
Since, for $j = k$, \eqref{eq:equiv cond} is equivalent to \eqref{eq:top condition}  it follows that Conditions \eqref{eq:equiv cond} hold. 

We have shown that  assertion (b) is true if $k \ge 2$. Using the first part of Theorem \ref{thm:Hilb-series}(b), one checks that (b) is also true if $k = 1$. 

Second, for  claim (a) one argues similarly.  We leave the details to the interested reader. 
\end{proof}

Part (a) of the previous result  implies: 

\begin{Corollary}
   \label{cor:regularity}
If $n \ge 2$, then
\[
\reg \cF(I_{\la - \mu}) =
\min \{n-1, \la_i - \mu_i + i - 3 \s 2 \le i \le n\}.
\]
\end{Corollary}

\begin{proof} Set $r = \min \{n-1, \la_i - \mu_i + i - 3 \s 2 \le i \le n\}$. Using  Corollary \ref{cor:non-vanishing}, we conclude that $h_r (\la - \mu) \neq 0$ and $h_{r+1} (\la - \mu) = 0$ because 
\[
r = \min \{n-1, \la_i - \mu_i + i - 3 \s 2 \le i \le r+2\}.
\] 
\end{proof}

Now we illustrate Corollary \ref{cor:non-vanishing}(b) in the case where $\overline{I}_{\lambda - \mu}$ is a strongly stable monomial ideal. 

\begin{Corollary}
    \label{cor:reg specialized fiber}
If $n \ge 2$ and $\mu = (0,1,\ldots,n-1) \in \ZZ^n$, then
\[
\reg \cF(\overline{I}_{\la - \mu}) =
\min \{n-1, \left \lfloor \frac{\la_i  + i}{2} \right \rfloor - 1 \s 2 \le i \le n\}.
\]    
\end{Corollary}

\begin{proof}
Using $\mu_i = i-1$, Corollary \ref{cor:non-vanishing}(b) gives $\overline{h}_k (\lambda - \mu)
  >0$ if and only if there is some integer $i_k \in \{k+1,\ldots,n\}$ such that 
\[
i_k \le \la_{i_k}
\]  
and
\begin{equation}
    \label{eq:cond spec}
k + i_k \le \la_i + i -1 \quad \text{ for all } i = 2,\ldots,k+1. 
\end{equation}

Set 
\[
r = \min \{n-1, \left \lfloor \frac{\la_i  + i}{2} \right \rfloor - 1 \s 2 \le i \le n\}.
\]

Now put $i_r = r+1$. Then $i_r \le n \le \la_n \le \la_{i_r}$ and, by definition of $r$, 
\[
r+i_r = 2 r+ 1 \le \la_i + i -1 
\]
for each $i = 2,\ldots,n$. Hence Conditions \eqref{eq:cond spec} are satisfied, and thus $\overline{h}_r (\lambda - \mu) >0$. This gives $\reg \cF(\overline{I}_{\la - \mu}) \ge r$. If $r = n-1$, then equality follows by Theorem \ref{thm:Hilb-series}. 

Assume $r \le n-2$. Then it remains to show that $\overline{h}_{r+1} (\lambda - \mu) = 0$.  
If $i \ge r+2$, then 
\[
\frac{\la_i + i}{2} \ge \frac{n+r + 2}{2} \ge r+2. 
\]
It follows that 
\[
r = \min \{n-1, \left \lfloor \frac{\la_i  + i}{2} \right \rfloor - 1 \s 2 \le i \le r+1\} = \frac{\la_j + j}{2} - 1
\]
for some $j \in \{2,\ldots,r+1\}$. This implies $\la_j + j \le 2 r + 3$. 

However, \eqref{eq:cond spec} with $k = r+1$ requires in particular
\[
2 r + 3 \le r+1 + i_{r+1} \le \la_{j} + j - 1. 
\]
This contradiction shows $\overline{h}_{r+1} (\lambda - \mu) = 0$, as desired. 
\end{proof}

Let us apply these results in order to compute the reduction number of the minimal reductions established in the previous section. There are results in the literature that relate reduction numbers and Castelnuovo-Mumford regularities under various assumptions (see, e.g., \cite{Trung}). However, we need the following observation. 

\begin{Proposition}
    \label{prop:red number - reg}
Let $I \subset R = K[x_1,\ldots,x_N]$ be a homogeneous ideal that is generated in one degree, say $d$, where $K$ is an infinite field.  Assume that the special fiber ring $\cF (I)$ is Cohen-Macaulay.  Then each minimal  reduction of $I$ is generated by $\dim \cF(I)$ homogeneous   polynomials of degree $d$,  and $I$ has reduction number 
\[
r (I) = \reg \cF(I). 
\]
\end{Proposition}

\begin{proof} For any minimal reduction $J$ of $I$, 
consider the equality $J I^k = I^{k+1}$, where integer $k = r_J(I)$. Since $J$ is contained in $I$, and $I$ is generated by homogeneous polynomials of degree $d$, the same must be true for $J$. 

As  $K$ is infinite each minimal reduction of $I$ is generated by $s = \dim \cF (I)$ elements.  Let $J = (g_1,\ldots,g_s)$ be such a reduction. The classes of its generators form a system of parameters of $\cF (I)$ that is linear. Since $\cF (I)$ is Cohen-Macaulay  $g_1,\ldots,g_s$ is a regular sequence. Regularity is invariant under quotient by a linear regular sequence (see, e.g., \cite[Lemma 2]{N-reg}). Thus, 
\[
\reg \cF (I) = \reg \cF (I)/ J \cF (I). 
\] 
As $\cF (I)/ J \cF (I)$ is artinian its regularity is determined by its largest non-vanishing degree component (see, e.g., \cite[Lemma 2.1]{N}), that is, 
\[
\reg \cF (I)/ J \cF (I) = \max\{ k \in \ZZ \; | \; [\cF (I)/ J \cF (I)]_k \neq 0\}.
\]
Notice that 
$
 [\cF (I)/ J \cF (I)]_k \cong I^k/(J I^{k-1} + \fm I^{k})$. 
Nakayama's Lemma implies that $[\cF (I)/ J \cF (I)]_k = 0$ if and only if $I^k = J I^{k-1}$. Now the claim follows.
\end{proof}

Finally, we are ready to determine the reduction number of any Ferrers ideal. 
Together with Theorem \ref{thm:min-red-Ferrers}, this completes our derivation of Dedekind-Mertens-like formulas in this case. 

\begin{Theorem}
     \label{thm:red number Ferrers}
For each partition $\lambda$, the reduction number of the Ferrers ideal $I_{\lambda}$   is  
\[
r  (I _{\lambda}) = \min \{n-1, \la_i + i - 3 \s 2 \le i \le n\}.
\]   
\end{Theorem} 

\begin{proof}
Consider $\mu = (\mu_1,\ldots,\mu_n)$, where $\mu_1 = \cdots = \mu_n = n$, and $\widetilde{\la} = (\la_1 + n,\ldots,\la_n + n)$. Then the special fiber rings of the ideals $I_{\la}, \ I_{\widetilde{\la} - \mu}$, and $\overline{I}_{\widetilde{\la} - \mu}$ are isomorphic. 

The Ferrers ideal $I_{\lambda}$ is generated in degree two, and its special fiber ring is Cohen-Macaulay (see, e.g., Theorem \ref{thm:special-fiber-ring-is-ladder-determinantal}). Hence Proposition \ref{prop:red number - reg} applies, and we conclude using Corollary \ref{cor:regularity} if $n \ge 2$. If $n =1$, then $J_{\la} = I_{\la}$, and thus $r_{J_{\lambda}} (I _{\lambda}) = 1$, completing the argument. 
\end{proof}

\begin{Example}
  \label{ex:complete bipartite}
If $\lambda = (m,m,\ldots,m) \in \ZZ^n$, then Theorem \ref{thm:red number Ferrers} gives 
$r (I _{\lambda}) =  \min\{m, n\} - 1$, and thus 
\[
J_{\lambda} \cdot I_{\lambda}^{\min\{m, n\} - 1} = I_{\lambda}^{\min\{m, n\}}. 
\]
This is the Dedekind-Mertens formula for the content of the product of two generic polynomials in \cite[Theorem 2.1]{CVV}, as discussed in the introduction. 
Theorems \ref{thm:min-red-Ferrers} and  \ref{thm:red number Ferrers} give analogous Dedekind-Mertens-like formulas with optimal exponents for an arbitrary Ferrers ideal. 
\end{Example}  
  
We now consider the specialized Ferrers ideals for which we found a distinguished minimal reduction in Theorem \ref{thm:min-red-specialized-Ferr}. 

\begin{Theorem}
   \label{thm:red number spec Ferrers}
Let $\mu = (0,1,\ldots,n-1) \in \ZZ^n$, and let $\la = (\la_1,\ldots,\la_n)$ be a partition, where $\la_1 = m$ and $\la_n \ge n$. Then the reduction number of the specialized  Ferrers ideal $\overline{I}_{\lambda - \mu}$   is 
\[
r (\overline{I}_{\lambda - \mu}) = \min \left \{n-1, \left \lfloor \frac{\la_i  + i}{2} \right \rfloor - 1  \; \big | \; 2 \le i \le n \right \}.
\]
\end{Theorem} 

\begin{proof}
The ideal $\overline{I}_{\lambda}$ is generated in degree two, and its special fiber ring is Cohen-Macaulay by Theorem \ref{thm:special-fiber-ring-is-ladder-determinantal}. Hence Proposition \ref{prop:red number - reg} and Corollary \ref{cor:reg specialized fiber} give the assertion if $n \ge 2$. If $n =1$, then $\overline{I}_{\lambda - \mu} = x_1 (x_{\mu_1 +1}\ldots,x_m)$, which is equal to each of its minimal reductions. This completes the argument. 
\end{proof}

We illustrate some of the above results in some very special cases. 

\begin{Example}
  \label{ex:Segre-embedding}
Let $2 \le n \le m $ be integers and consider the partition $\lambda
= (m,\ldots,m) \in \ZZ^n$ and $\mu = (0,1,\ldots,n-1) \in \ZZ^n$. Then the
coefficients of the normalized numerator in the Hilbert series of the toric
ring $\cF(\overline{I}_{\lambda - \mu})$ are
\begin{eqnarray*}
	\overline{h}_1 (\lambda - \mu ) & = &  \max\{0,n-1\}+\sum_{j=2}^n(m-(j-1)-1) =  n-1+\sum_{j=2}^n(m-j), \\
& = & m (n-1) - \binom{n}{2}. 	
\end{eqnarray*}
and, for $k\geq 2$,
\begin{align*}
	\overline{h}_k (\lambda - \mu ) &= \sum_{2\leq i_1<i_2<\dots<i_k\leq n} \ \sum_{j_{k-1}=-k+2}^{m-i_k-k+2}\ \sum_{j_{k-2}=-k+3}^{j_{k-1}}\dots \sum_{j_1=0}^{j_2}j_1 \\
	&= \sum_{k<i_k\leq n} {{i_k-2}\choose {k-1}} \ \sum_{j_{k-1}=-k+2}^{m-i_k-k+2}\ \sum_{j_{k-2}=-k+3}^{j_{k-1}}\dots \sum_{j_1=0}^{j_2}j_1 \\
	&= \sum_{k<i_k\leq n} {{i_k-2}\choose {k-1}} \ \sum_{j_{k-1}=-k+2}^{m-i_k-k+2}\ \sum_{j_{k-2}=-k+3}^{j_{k-1}}\ \sum_{j_{k-3}=-k+4}^{j_{k-2}} \dots \sum_{j_2=-1}^{j_3} \binom{j_2 + 1}{2} \\
& = \cdots  \\ 
& = \sum_{k<i_k\leq n} {{i_k-2}\choose {k-1}} \binom{m - i_k + 1}{k}. 
\end{align*}
Notice that, for a fixed $k$, this is a sum  over a polynomial in $i_k$ of degree $2k-1$, which can be evaluated explicitly. For example, if $k = 2$, then 
\[
\overline{h}_2 (\lambda - \mu ) = \binom{m}{4} - m \binom{m+1 -n}{3} + 3 \binom{m+2-n}{4}. 
\]
However, a general formula does not seem to be easy, except in the case  $m = n$. 

Indeed, if $m = n$, then the above formulae simplify to give, for all $k \ge 0$, 
\[
	\overline{h}_k (\lambda -\mu) =  \binom{n}{2 k}. 
\]
In this special case there is a more direct approach. Observe that if $m = n$, then $\overline{I}_{\lambda - \mu} = (x_1,\ldots,x_n)^2$,
thus $\cF(\overline{I}_{\lambda - \mu})$ is the coordinate ring of the
second Veronese embedding of $\PP^{n-1}$ into $\PP^{\binom{n+1}{2} - 1}$. Hence the Hilbert function in non-negative degrees is 
\[
h_{K[\overline{I}_{\lambda - \mu}]} (j) = \binom{n-1+2j}{n-1}.
\]
The multiplicity is $e( \cF(\overline{I}_{\la - \mu})) = 2^{n-1}$.  Observe that $\cF(\overline{I}_{\la - \mu})$ is a Gorenstein ring if and only if $n$ is even. 
\end{Example}

Now consider the analogous squarefree specialized Ferrers ideals.

\begin{Example}
  \label{ex:hypersimplex}
Let $2 \le n < m $ be integers and consider the partition $\lambda
= (m,\ldots,m) \in \ZZ^n$ and $\mu = (1,2,\ldots,n) \in \ZZ^n$.
Then, 
\begin{eqnarray*}
	\overline{h}_1 (\lambda - \mu ) & = &  \max\{0,n-2\}+\sum_{j=2}^n(m-j-1) = n-2+\sum_{j=2}^n(m-j-1) \\
& = & m (n-1) - \binom{n+1}{2} 
\end{eqnarray*}
and, for $k\geq 2$,
\begin{align*}
	\overline{h}_k (\lambda - \mu ) &= \sum_{2\leq i_1<i_2<\dots<i_k\leq n}\ \sum_{j_{k-1}=-k+3}^{m-i_k-k+2}\ 
\sum_{j_{k-2}=-k+3}^{j_{k-1}}\ \sum_{j_{k-3}=-k+4}^{j_{k-2}} \dots \sum_{j_1=0}^{j_2} j_1\\
&= \sum_{k<i_k\leq n} {{i_k-2}\choose {k-1}} \ \sum_{j_{k-1}=-k+3}^{m-i_k-k+2}\ \sum_{j_{k-2}=-k+3}^{j_{k-1}}\ \sum_{j_{k-3}=-k+4}^{j_{k-2}} \dots \sum_{j_1=0}^{j_2} j_1 \\
&= \sum_{k<i_k\leq n} {{i_k-2}\choose {k-1}} \ \sum_{j_{k-1}=-k+3}^{m-i_k-k+2}\ \sum_{j_{k-2}=-k+3}^{j_{k-1}}\ \sum_{j_{k-3}=-k+4}^{j_{k-2}} \dots \sum_{j_2=-1}^{j_3} \binom{j_2 + 1}{2} \\
& = \cdots  \\ 
&= \sum_{k<i_k\leq n} {{i_k-2}\choose {k-1}} \ \sum_{j_{k-1}=-k+3}^{m-i_k-k+2}  \binom{j_{k-1} + k- 2}{k-1} \\
& = \sum_{k<i_k\leq n} {{i_k-2}\choose {k-1}} \binom{m - i_k + 1}{k}, 
\end{align*}
where in the formula from Theorem~\ref{thm:Hilb-series} we used that $i_k>i_j$ implies $i_k-1\geq i_j$,  for $j=2,\dots,k-1$.

Note that in the special case $m = n+1$ these formulae are again well-known.
Indeed, then $\overline{I}_{\lambda - \mu}$ is the edge ideal of a complete graph on $n+1$ vertices, and the result simplifies to
\begin{align*}
  \overline{h}_k (\lambda - \mu) & = \begin{cases}
   {\displaystyle  \binom{n+1}{2 k}} & \text{if } k \le \frac{n+1}{2} \text{ and } k \neq 1 \\[3ex]
    {\displaystyle \frac{(n + 1) (n-2)}{2}} & \text{if } k = 1 
  \end{cases},
\end{align*}
as first shown in \cite[Remark 9.2.11]{vila}. For the multiplicity, we obtain: 
\begin{eqnarray*}
e( \cF(\overline{I}_{\la - \mu})) & = &   \overline{h}_0 (\lambda - \mu) + \cdots +  \overline{h}_{n-1} (\lambda - \mu) \\
& = & \overline{h}_1 (\lambda - \mu)  - \binom{n+1}{2} + \sum_{k \ge 0} \binom{n+1}{2 k} \\[1ex]
& = &  2^{n}  - (n+1). 
\end{eqnarray*}

In the  case $m = n+1$, also the Hilbert function of $\cF(\overline{I}_{\lambda - \mu})$ admits a nice form. Indeed, in non-negative degrees it equals its Hilbert polynomial, which in turn is equal to the Ehrhart polynomial of the second hypersimplex  in $\mathbb R^{n+1}$ with $\binom{n+1}{2}$ vertices (see \cite[Corollary 9.6]{St}):
\[
h_{\cF(\overline{I}_{\lambda - \mu})} (j) = \binom{n+2j}{n} - n \binom{n+j-1}{n}
\]
for all $j \geq 0$. Note that the multiplicity of $\cF(\overline{I}_{\lambda - \mu})$, that is $2^n - (n+1)$,  is the volume of this hypersimplex. 
\end{Example}

\newpage

{
\section{Final remarks and open problems}\label{sec:Problems}}
\subsection{Shapes of  minimal free resolutions}
{
In Corollary \ref{cor:primeness} we showed in particular that the rings $K[\bTlm]/I_2 (\bTlm)$ and $K[\bTlm]/I_2 (\bSlm)$ are Koszul rings. This means that the minimal graded free resolution of the residue field over the original ring is linear, that is, it is of the form 
\[
\cdots \to A^{\gamma_2} (-2) \to A^{\gamma_1}(-1) \to A \to K \to 0, 
\]
where $A$ is $K[\bTlm]/I_2 (\bTlm)$ or $K[\bTlm]/I_2 (\bSlm)$ and $K$ is considered as the quotient of $A$ by its maximal homogeneous ideal. It is a very interesting question to study the shape of the graded minimal free resolution of $A$ over the polynomial ring $R = K[\bTlm]$. The first problem is then to describe the length of the linear part of the resolution. This question has been mostly investigated in the case when the ideal is generated by quadrics. In fact, one says that $A = R/I$ has property $(N_p)$ for some integer $p \ge 1$ if the first $p$ steps in the resolution are linear (see \cite{GL}). More precisely, the graded minimal free resolution has the form 
\[
\cdots \to F_{p+1} \to R^{\beta_p} (-p-1)\to \cdots  \to R^{\beta_1}(-2) \to R \to A \to 0. 
\]
Hence property $(N_1)$ means that the ideal $I$ is generated by quadrics. The algebra $A$ has property $(N_2)$ if in addition all first syzygies of $I$ are linear. There is a rich literature investigating property $(N_p)$ in various cases, initiated in \cite{Green}. 
}

{
Notice that the resolution of $A$ is linear if and only if its Castelnuovo-Mumford regularity is one. By Corollaries \ref{cor:regularity} and  \ref{cor:reg specialized fiber}, the regularity of $K[\bTlm]/I_2 (\bTlm)$ and $K[\bTlm]/I_2 (\bSlm)$ is greater than one in most cases. Thus, it would be very interesting to establish results on the length of the linear part of the resolution of these rings. This is a  challenging problem. It is even open in the very special case, where  $\lambda
= (m,\ldots,m) \in \ZZ^n$, \ $\mu = (0,1,\ldots,n-1) \in \ZZ^n$, and $2 \le n \le m$ as considered in Example \ref{ex:Segre-embedding}. If we further assume $m = n$, then $A = K[\bTlm]/I_2 (\bSlm)$ is the coordinate ring of the second Veronese embedding of $\PP^{n-1}$. In this latter case, the resolution is linear if and only if $n \le 3$. If $n \ge 4$ and the characteristic of $K$ is zero, then $A$ has property $(N_5)$, but not $(N_6)$ by \cite{JPW}. However, if $n \ge 6$ and the characteristic of $K$ is five, then $A$  satisfies $(N_4)$, but not $(N_5)$ by \cite{A}. 
}

\subsection{Finding explicit minimal reductions}
{
In Theorem \ref{thm:min-red-specialized-Ferr} we determined a distinguished minimal reduction of the ideal $\overline{I}_{\la - \mu}$ in an important case. It would be desirable to extend this result to further cases as this could be useful, for example,  in studying the core of Ferrers ideals or in investigations in algebraic statistics (see below). The challenge is that then the number of generators of a minimal reduction is greater than the number of diagonals in $\bTlm$. An interesting first step would be to settle the case when  the difference is one. We note that,  experimentally, the pattern for a minimal reduction as found in Example \ref{ex:min-red} seems to generalize. 
}

\subsection{Other shapes of matrices and higher minors}
{
The determinantal rings $K[\bTlm]/I_2 (\bTlm)$ and $K[\bTlm]/I_2 (\bSlm)$ arise naturally as special fiber rings of Ferrers and specialized Ferrers ideal (see \cite[Proposition 5.1]{CorsoNagel}  and Theorem \ref{thm:special-fiber-ring-is-ladder-determinantal}) or as Rees algebras (see Remark \ref{rem:Rees-algebra-simplified}). However, one may also  view these determinantal rings as obtained from a generic (symmetric) matrix, where some variables are forbidden for use in any minor. Alternatively, one could replace the ``forbidden'' variables by zeros and eliminate forbidden subregions of the matrix. It is shown in \cite{B} that the ideal generated by the  maximal minors of such a matrix always has a linear resolution and that the non-zero minors form a universal Gr\"obner basis. 
}

{
It would be interesting to investigate ideals  generated by higher minors of a tableau $\bTlm$ or $\bSlm$. Steps in this direction are taken in \cite{NRo}. 
}

\subsection{Connections with algebraic statistics}

{
We hope that our results also motivate further investigations in applications of monomial algebras. For instance,  special fiber rings of edge ideals of graphs (and hypergraphs) make a notable appearance in algebraic statistics, which considers statistical models with rational parametrizations. Specifically, the structure of the model of independence of  two  categorical random variables $X\in[n]$ and $Y\in[m]$ is such that it is parametrized by  the edges of a bipartite graph on $n$ and $m$ vertices. In symbols, $P(X=i,Y=j)=P(X=i)P(Y=j)$, where $P(A)$ denotes the probability of the event $A$. If we denote the marginal probabilities by $P(X=i)=x_i$ and $P(Y=i)=y_i$, the edge $x_iy_j$ in a bipartite graph corresponds to the probability of the joint state of the random variables $(X,Y)=(i,j)$. 
Without any restrictions on the model, the bipartite graph is complete. Missing edges in the bipartite graph correspond to the so-called \emph{structural zeros} in the statistical model, that is, joint states of random variables that are simply unobservable. In this sense, \emph{Ferrers graphs correspond to independence models where structural zeros appear  with a hierarchy}: if the joint state $(i,j)$ is not a zero, then neither are any states $(k,l)$ with $k<i$ and $l<j$. 
In the same way, the generalized Ferrers tableau with $\mu_i\geq i$ parametrizes the model of quasi-independence with structural zeros, obtained from the independence model by removing the diagonals of the tableau. 
}

{
One of the early results in the field is the Fundamental Theorem of Markov Bases \cite{DS98}, which states that a Markov basis for a  log-linear statistical model on discrete variables, of which the independence models are quintessential examples, is given by a generating set of a corresponding toric ideal. In algebraic language, a Markov basis of a model parametrized by monomials is a generating set of the special fiber ring of that monomial ideal. 
 A Markov basis is necessary for testing fit of a proposed model to the given data, and the Fundamental Theorem applies to a large class of models used in practice. However, most of the time, determining the Markov basis - theoretically or using a computer -  is a highly non-trivial task due to the size of the problems that arise in applications. In other situations, structural zeros pose a significant challenge, in that many of the elements of Markov basis are not applicable, as they attempt to place observations where structural zeros disallow them. In such cases, having an explicit description of a Gr\"obner basis can be important for more efficient computation, see for example \cite{Rapallo06} or \cite{KateriTablesBook}. 
}

{
Independence models are not the only ones encoded by graphs; there are other families of models whose building blocks have bipartite graph structure.  
 Recent results, e.g., \cite{PRF10, GPS14+},  show that for some very popular models of random graphs, Markov bases can be constructed by appropriately composing generators of the edge subring of a bipartite graph. 
The  problem of the computational difficulty of Markov bases has  been addressed  for various models recently \cite{RY10, HAT11, GPS14+,Uhler:IsingModel16}  by considering subsets of generators that can be applied to given data. 
Another direction of interest is how well Markov chains based on various types of bases behave; the interested readers should see, e.g., \cite[Section 3.1]{PetReview16} for references to the literature on mixing times. 
It is possible that using a nice Gr\"obner basis leads to positive results, though it seems more likely that Markov chains based on graph-theoretic sampling algorithms could perform better. These types of questions are subjects of ongoing research. 
}

{
More interestingly, a connection has not yet  been made between statistical concepts and reductions, reduction numbers, or other similar information. Generators of a minimal reduction are sums of monomials from the model parametrization. Since minimal reductions carry a lot of algebraic information about the original monomial ideal, and  algebraic information such as dimension captures the algebraic complexity of the model, an open question is to investigate how to use minimal reductions in statistical modeling, sampling, or inference. 
}


\end{document}